\documentclass{article}
\usepackage[margin= 1in]{geometry}
\usepackage{times}
\usepackage{graphicx} 
\usepackage{subfigure} 
\usepackage{epsfig}
\usepackage{natbib}
\usepackage{color}

\usepackage{algorithm}
\usepackage{algorithmic}

\usepackage{hyperref}

\usepackage{amsmath}
\usepackage{amsthm}
\usepackage{amssymb}
\usepackage{amsfonts}
\usepackage{mathtools,xparse}
\usepackage{epstopdf}
\usepackage{multirow}

\newcommand{\mb}[1]{\mathbf{#1}}
\newcommand{\mc}[1]{\mathcal{#1}}
\newcommand{\mr}[1]{\mathbb{#1}}

\DeclarePairedDelimiter{\norm}{\lVert}{\rVert}
\NewDocumentCommand{\normL}{ s O{} m }{%
  \IfBooleanTF{#1}{\norm*{#3}}{\norm[#2]{#3}}_{L_2(\Omega)}%
}
\def\k{{(k)}}
\newtheorem{Lemma}{Lemma}

\newtheorem{Corollary}{Corollary}
\newtheorem{Theorem}{Theorem}
\title{Linear Convergence of Stochastic Frank Wolfe Variants}
\author{Donald Goldfarb, Garud Iyengar, and Chaoxu Zhou}
\date{October 12, 2016}
\begin{document}
\maketitle
\begin{abstract}
\noindent In this paper, we show that the Away-step Stochastic Frank-Wolfe Algorithm (ASFW) and Pairwise Stochastic Frank-Wolfe algorithm (PSFW) converge linearly in expectation. We also show that if an algorithm convergences linearly in expectation then it converges linearly almost surely.  In order to prove these results, we develop a novel proof technique based on concepts of empirical processes and concentration inequalities. Such a technique has rarely been used to derive the convergence rates of stochastic optimization algorithms. In large-scale numerical experiments, ASFW and PSFW perform as well as or better than their stochastic competitors in actual CPU time.
\end{abstract}
\section{Introduction}
\textbf{Motivation.} The recent trend of using a large number of parameters to model large datasets in machine learning and statistics has created a strong demand for optimization algorithms that have low computational cost per iteration and exploit model structure. Regularized empirical risk minimization (ERM) is an important class of problems in this area that can be formulated as smooth constrained optimization problems. A popular approach for solving such ERM problems is the proximal gradient method which solves a projection sub-problem in each iteration. The major drawback of this method is that the projection step can be expensive in many situations. As an alternative, the Frank-Wolfe (FW) algorithm \citep{FW56}, also known as the conditional gradient method, solves a linear optimization sub-problem in each iteration, which is much faster than the standard projection technique when the feasible set is a simple polytope \citep{N15}. When the number of observations in ERM is large, calculating the gradient in every FW iteration becomes a computationally intensive task. The question of whether `cheap' stochastic gradients can be used as a surrogate in FW immediately arises.\\
\noindent\textbf{Contribution.} In this paper, we show that the Away-step Stochastic Frank-Wolfe (ASFW) algorithm converges linearly in expectation and each sample path of the algorithm converges linearly. We also show that if an algorithm converges linearly in expectation then it converges linearly almost surely. To the best of our knowledge, this is the first paper that proves these results. The major technical difficulty of analyzing the ASFW algorithm is the lack of tools that combine stochastic arguments and combinatorial arguments. In order to solve this problem and prove our convergence results, a novel proof technique based on concepts in empirical processes theory and concentration inequalities is developed.  This technique is then applied to prove the linear convergence in expectation and almost sure convergence of each sample path of another Frank-Wolfe variant, the Pairwise Stochastic Frank-Wolfe (PSFW) algorithm. We note that this technique may be useful for analyzing the convergence of other stochastic algorithms. In our large-scale numerical experiments, the proposed algorithms outperform their competitors in all different settings.  \\
\noindent\textbf{Related Work.} The Frank-Wolfe algorithm was proposed sixty years ago  \citep{FW56} for minimizing a convex function over a polytope and is known to converge at an $O(1/k)$ rate. In \cite{LP66} the same convergence rate was proved for compact convex constraints. When both objective function and the constraint set are strongly convex, \cite{GH15} proved that the Frank-Wolfe algorithm has an  $O(1/k^2)$ rate of convergence with a properly chosen step size.  Motivated by removing the influence of ``bad" visited vertices, the away-steps variant of the Frank-Wolfe algorithm was proposed in \cite{Wol70}. Later, \cite{GM86} showed that this variant converges linearly under the assumption that the objective function is strongly convex and the optimum lies in the interior of the constraint  polytope. Recently, \cite{GH13} and \cite{LJJ13} extended the linear convergence result by removing the assumption of the location of the optimum and \cite{BS15} extended it further by relaxing the strongly convex objective function assumption. Stochastic Frank-Wolfe algorithms have been considered by \cite{LAN13} and \cite{LWM15} in which an $O(1/k)$ rate of convergence in expectation is proved. \cite{LH16} considered the Stochastic Varianced-Reduced Frank-Wolfe method (SVRF) which also has convergence rate $O(1/k)$ in expectation. In addition, the Frank-Wolfe algorithm has been applied to solve several different classes of problems, including non-linear SVM \citep{OG10}, structural SVM \citep{LJJSP13} \citep{OALDLJ16}, and comprehensive principal component pursuit \citep{MZWG15} among many others. To compare FW variants and other useful algorithms such as the Prox-SVRG of \cite{LZ14} and the stochastic variance reduced FW algorithm of \cite{LH16}, we summarize the theoretical performance in Table 1 which includes the required conditions for convergence and the given complexity bounds, the number of exact and stochastic gradient oracle calls, the number of linear optimization oracle (LO) calls and the number of projection calls in order to obtain an $\epsilon$-approximate solution.

\begin{table}
\begin{center}
\begin{tabular}{|c|c|c|c|c|c|}
\hline
 \textbf{Algorithm} & \textbf{Extra conditions} & \textbf{Exact gradients} & \textbf{Stochastic Gradients} & \textbf{LO} & \textbf{Projection}\\
 \hline
 FW & bounded constraint &$O(1 /\epsilon)$ & NA & $O(1  /\epsilon)$ & NA \\
 \hline
 Away-step       & polytope constraint & &  &  & \\
 FW    &strongly convex objective & $O(\log(1 / \epsilon))$&NA &$O(\log(1 / \epsilon))$ & NA \\
 \hline
 Pairwise& polytope constraint & &  &  & \\
 FW & strongly convex objective &$O(\log(1 / \epsilon))$ & NA & $O(\log(1 / \epsilon))$ & NA \\
 \hline
 SVRF & bounded constraint &$O(\log(1 / \epsilon))$ & $O(1/\epsilon^2)$ & $O(1 / \epsilon)$ & NA \\
 \hline
 Prox-SVRG &strongly convex objective & $\log(1 / \epsilon)$ & $O(m\log(1 / \epsilon))$ &NA & $O(m\log(1 / \epsilon))$ \\
 \hline
  ASFW & polytope constraint   &   &$O(1 / \epsilon^{4\eta}),$       & & \\
   & strongly convex objective &NA  &$ 0 < \eta < 1$    &$O(\log 1 / \epsilon)$ & NA\\
 \hline
 PSFW & polytope constraint   &    &$O(1 / \epsilon^{(6\vert V \vert ! + 2)\zeta}),$ & & \\
  &strongly convex objective  &NA  & $ 0 < \zeta < 1$ &$O(\log 1 / \epsilon)$ & NA\\
 \hline
\end{tabular}
\caption{Comparisons of algorithms in terms of their requirements and the theoretical performances to get an $\epsilon$-approximate solution. LO denotes for linear optimizations and. In Prox-SVRG, $m$ is the number of iterations in each epoch. In PSFW, $\vert V \vert$ is the number of vertices of the polytope constraint.}
\end{center}
\end{table}

\noindent\textbf{Problem description.} Consider the minimization problem
\begin{align}
\min_{\mb{x} \in \mc{P}} \Big\{F(\mb{x}) \equiv \frac{1}{n}\sum_{i = 1}^n f_i(\mb{x}) \Big\}, \label{general_problem} \tag{P1}
\end{align}
where $\mathcal{P}$ is a polytope, i.e., a non-empty compact polyhedron given by $\mathcal{P} =  \{\mb{x} \in \mathbb{R}^p : \mathbf{Cx} \leq \mb{d}\}$ for some $\mathbf{C} \in \mathbb{R}^{m \times p}$, $\mathbf{d} \in \mathbb{R}^m$. Therefore, the set of vertices $V$ of the polytope $\mc{P}$  has finitely many elements. Let $D = \sup\{\norm{\mb{x} - \mb{y}} \ \vert \ \mb{x}, \mb{y} \in \mc{P}\}$ be the diameter of $\mc{P}$. For every $i = 1, \ldots, n$, $f_i : \mr{R} \rightarrow \mr{R}$ is a strongly convex function with parameter $\sigma_i$ with an $L_i$ Lipschitz continuous gradient. From another point of view, \ref{general_problem} can be reformulated as a stochastic optimization problem as below
\begin{align}
\min_{\mb{x} \in \mc{P}} \Big\{\frac{1}{n}\sum_{i = 1}^n f_i(\mb{x}) \equiv \mr{E} f(\xi, \mb{x}) \Big\} \label{stochastic_problem} \tag{SP1}
\end{align}
where $\xi$ is a random variable that follows a discrete uniform distribution on $\{1, \ldots, n\}$, $f(i, \mb{x}) = f_i(\mb{x})$ for every $i = 1, \ldots, n$ and $\mb{x} \in \mc{P}$. Furthermore, define $\nabla f(\xi, \mb{x}) = \nabla f_\xi(\mb{x})$.\\
\noindent\textbf{The Frank-Wolfe Algorithm and its variants.} In contrast to the projected gradient algorithm, the Frank-Wolfe algorithm (also known as conditional gradient algorithm) calls a linear optimization oracle instead of a projection oracle in every iteration. 
\begin{algorithm}[H]
\caption{The Frank-Wolfe Algorithm}
\label{frank_wolfe_algo}
\begin{algorithmic}
\STATE {\bfseries Input:}  $\mathbf{x}^{(1)} \in \mc{P}$, $F(\cdot)$
\FOR{$k = 1, 2, \ldots$}
\STATE Set $\mb{p}^\k = \arg\min_{\mb{s} \in \mc{P}} \langle \nabla F(\mb{x}^\k), \mb{s}\rangle$.
\STATE Set $\mb{d}^\k = \mb{p}^\k - \mb{x}^\k$.
\STATE Set $\mb{x}^{(k+1)} = \mb{x}^\k + \gamma^\k\mb{d}^\k$, where $\gamma^\k = \frac{2}{k+2}$ or obtain by line-search.
\ENDFOR
\STATE{\bfseries Return:} $\mathbf{x}^{(k+1)}$.
\end{algorithmic}
\end{algorithm}
\noindent The Frank-Wolfe Algorithm has become popular recently because it performs a sparse update at each step. For a good review of what was known about the FW algorithm until a few years ago, see \cite{JAGGI13}. It is well-known that this algorithm converges sub-linearly with rate $O(1 / k)$ because of the so-called zig-zagging phenomenon \citep{LJJ15}. Especially when the optimal solution $\mb{x}^*$ does not lie in the relative interior of $\mc{P}$, the FW algorithm tends to zig-zag amongst the vertices that define the facet containing $\mb{x}^*$. One way to overcome this zig-zagging problem is to keep tracking of the "active`` vertices (the vertices discovered previously in the FW algorithm) and move away from the ``worst'' of these in some iterations.

\noindent The Away-step Frank-Wolfe algorithm (AFW) and the Pairwise Frank-Wolfe algorithm (PFW) are two notable variants based on this idea. After computing the vertex $\mb{p}^\k = \arg\min_{\mb{x} \in \mc{P}} \langle \nabla F(\mb{x}^\k), \mb{x} \rangle$ by the linear optimization oracle and the vertex $\mb{u}^\k = \arg\max_{\mb{x} \in U^\k} \langle \nabla F(\mb{x}^\k), \mb{x} \rangle $ where $U^\k$ is the set of active vertices at iteration $k$, the AFW algorithm moves away from the one that maximizes the potential increase in $F(\mb{x})$ i.e. the increase in the linearized function, while the PFW algorithm tries to take advantages of both vertices and moves in the direction $\mb{p}^\k - \mb{u}^\k$. Details of the algorithms can be found in \cite{LJJ15}.

\section{Variants of Stochastic Frank-Wolfe Algorithm}
When the exact gradients is expensive to compute and an unbiased stochastic gradient is easy to obtain, it may be advantageous to use a stochastic gradient in AFW and PFW. We describe the Away-step Stochastic Frank-Wolfe Algorithm (ASFW) and the Pairwise Stochastic Frank-Wolfe Algorithm(PSFW) below.
\begin{algorithm}[H]
\caption{Away-step Stochastic Frank-Wolfe algorithm}
\label{cond_grad_1}
\begin{algorithmic}[1]
\STATE {\bfseries Input:}  $\mathbf{x}^{(1)} \in V$, $f_i$ and $L_i$
\STATE Set $\mu^{(1)}_{\mathbf{x}^{(1)}} = 1$, $\mu_\mathbf{v}^{(1)} = 0$ for any $\mathbf{v} \in {V} / \{\mathbf{x}^{(1)}\}$ and $U^{(1)} = \{\mathbf{x}^{(1)}\}$.
\FOR{$k = 1, 2, \ldots$}
\STATE Sample $\xi_1, \ldots, \xi_{m^\k} \overset{\text{i.i.d.}}{\sim} \xi$ and set $\mb{g}^\k = \frac{1}{m^\k}\sum_{i=1}^{m^\k}\nabla_\mb{x}f(\xi_i, \mb{x}^\k)$, $L^\k = \frac{1}{m^\k}\sum_{i=1}^{m^\k}L_{\xi_i}$.
\STATE Compute $\mathbf{p}^{(k)} \in \arg\min_{\mb{x} \in \mc{P}}\langle \mathbf{g}^{(k)}, \mb{x} \rangle$.
\STATE Compute $\mathbf{u}^{(k)} \in {\arg\!\max}_{\mathbf{v} \in U^{(k)}}\langle \mathbf{g}^{(k)}, \mathbf{v} \rangle$.
\IF{$\langle \mathbf{g}^{(k)}, \mathbf{p}^{(k)}  + \mathbf{u}^{(k)}- 2\mathbf{x}^{(k)} \rangle \leq 0$}
\STATE Set $\mathbf{d}^{(k)} = \mathbf{p}^{(k)} - \mathbf{x}^{(k)}$ and $\gamma^{(k)}_{\text{max}} = 1$. 
\ELSE
\STATE Set $\mathbf{d}^{(k)} = \mathbf{x}^{(k)} - \mathbf{u}^{(k)}$ and $\gamma^{(k)}_{\text{max}} = \frac{\mu^{(k)}_{\mathbf{u}^{(k)}}}{1 - \mu^{(k)}_{\mathbf{u}^{(k)}}}$.
\ENDIF
\STATE Set $\gamma^{(k)} = \min\{-\frac{\langle \mb{g}^{(k)}, \mb{d}^{(k)}\rangle}{L^\k\norm{\mb{d}^{(k)}}^2}, \gamma^{(k)}_\text{max}\}$ or determine it by line-search.
\STATE Set $\mathbf{x}^{(k+1)} = \mathbf{x}^{(k)} + \gamma^{(k)}\mathbf{d}^{(k)}$.
\STATE Update $U^{(k + 1)}$ and $\mathbf{\mu}^{(k+1)}$ by Procedure VRU.
\ENDFOR
\STATE{\bfseries Return:} $\mathbf{x}^{(k+1)}$.
\end{algorithmic}
\end{algorithm}
\begin{algorithm}[H]
\caption{Pairwise Stochastic Frank-Wolfe algorithm}
\label{cond_grad_pair}
\begin{algorithmic}[1]
\STATE Replace line 7 to 11 in Algorithm \ref{cond_grad_1} by: $\mb{d}^\k = \mb{p}^\k - \mb{u}^\k$ and $\gamma^\k_{\max} = \mu^\k_{\mb{u}^\k}$.
\end{algorithmic}
\end{algorithm}
\noindent The following algorithm updates a vertex representation of the current iterate and is called in Algorithms \ref{cond_grad_1} and \ref{cond_grad_pair}.
\begin{center}
\begin{algorithm}[H]
\caption{Procedure Vertex Representation Update (VRU)}
\begin{algorithmic}[1]
\STATE {\bfseries Input:} $\mb{x}^\k$, $(U^\k, \boldsymbol\mu^\k)$, $\mb{d}^\k$, $\gamma^\k$, $\mb{p}^\k$ and $\mb{v}^\k$.
\IF{$\mb{d}^\k = \mb{x}^\k - \mb{u}^\k$}
\STATE Update $\mu^\k_\mb{v} = \mu_\mb{v}^\k(1 + \gamma^\k)$ for any $\mb{v} \in U^\k / \{\mb{u}^\k\}$.
\STATE Update $\mu^{(k+1)}_{\mb{u}^\k} = \mu^\k_{\mb{u}^\k}(1 + \gamma^\k) - \gamma^\k$.
\IF{$\mu^{(k+1)}_{\mb{u}^\k} = 0$}
\STATE Update $U^{(k+1)} = U^\k /\{\mb{u}^\k\}$
\ELSE
\STATE Update $U^{(k+1)} = U^\k$
\ENDIF
\ENDIF
\STATE Update $\mu_\mb{v}^{(k+1)} = \mu^\k_\mb{v}(1 - \gamma^\k)$ for any $\mb{v} \in U^\k / \{\mb{p}^\k\}$.
\STATE Update $\mu^{(k+1)}_{\mb{p}^\k} = \mu^\k_{\mb{p}^\k}(1 - \gamma^\k) + \gamma^\k$.
\IF{$\mu_{\mb{p}^\k}^{(k+1)} = 1$}
\STATE Update $U^{(k+1)} = \{\mb{p}^\k\}$.
\ELSE
\STATE Update $U^{(k+1)} = U^\k \cup \{\mb{p}^\k\}$.
\ENDIF
\STATE (Optional) Carath\'eodory's theorem can be applied for the vertex representation of $\mb{x}^{(k+1)}$ so that $\vert U^{(k+1)}\vert = p+1$ and $\boldsymbol\mu^{(k+1)} \in \mr{R}^{p+1}$.
\STATE {\bfseries Return:} $(U^{(k+1)}, \boldsymbol\mu^{(k+1)})$
\end{algorithmic}
\end{algorithm}
\end{center}
\section{Convergence Proof}
In this section, we will first introduce some lemmas and notation and then prove the main theorems in this paper. Note that, at the $k$-th iteration of the algorithms, $m^\k$ i.i.d. samples of $\xi$ are obtained. Define $ F^\k(\mb{x}) = \frac{1}{m^\k}\sum_{i=1}^{m^\k} f_{\xi_i}(\mb{x})$. It is easy to see that $F^\k$ is Lipschitz continuous with Lipschitz constant $L^\k = \frac{1}{m^\k}\sum_{i=1}^{m^\k}L_{\xi_i}$ and strongly convex with constant $\sigma^\k = \frac{1}{m^\k}\sum_{i=1}^{m^\k}\sigma_{\xi_i}$. The following ancillary problem is used in our analysis.
\begin{align}
\min_{\mb{x} \in \mc{P}} F^\k(\mb{x}), \label{H_problem} \tag{H1}
\end{align}
Let $\mb{x}^\k_*$ denote the optimal solution of problem \ref{H_problem}, i.e., $\mb{x}^\k_* = \arg\!\min_{\mb{x} \in \mc{P}} F^\k(\mb{x})$. The lemma below plays an important role in our proof. We refer to \cite{BS15} for a detailed proof of this lemma.
\begin{Lemma}\label{linear_approx_lower_bound}
For any $\mathbf{x} \in \mathcal{P} / \{\mb{x}^\k_*\}$ that can be represented as $\mathbf{x} = \sum_{\mathbf{v} \in U^\k}\mu_{\mathbf{v}}\mathbf{v}$ for some $U^\k \subset V$ where $\sum_{\mb{v} \in U^\k} \mu_{\mb{v}} = 1$ and $\mu_{\mb{v}} > 0$ for every $\mb{v} \in U^\k$, it holds that, 
\begin{align*}
\max_{\mathbf{u} \in U, \mathbf{p} \in V}\langle \nabla F^\k(\mathbf{x}), \mathbf{u - p} \rangle \geq \frac{\Omega_{\mathcal{P}}}{\vert U \vert} \frac{\langle \nabla F^\k(\mathbf{x}), \mathbf{x} - \mathbf{x}^\k_* \rangle}{\norm{\mathbf{x} - \mb{x}^\k_*}}.
\end{align*}
where $\vert U^\k \vert$ denotes the cardinality of $U^\k$, $V$ is the set of extreme points of $\mc{P}$ and
\begin{align*}
 \Omega_\mathcal{P} = \frac{\zeta}{\phi}
\end{align*}
for
\begin{align*}
\zeta &= \min_{\mathbf{v} \in V, i \in \{1, \ldots, m\}: a_i > \mathbf{C}_i \mathbf{v}} (d_i - \mathbf{C}_i\mathbf{v}), \\
\phi &= \max_{i \in \{1, \ldots, m\} / I(V)}\norm{\mathbf{C}_i}.
\end{align*}
\end{Lemma}
\noindent Next, we introduce some definitions and lemmas that are common in the empirical processes literature but rarely seen in the optimization literature.\\

\noindent\textbf{Definition} [Bracketing Number] Let $\mc{F}$ be a class of functions. Given two functions $l$ and $u$, the bracket $[l, u]$ is the set of all function $f$ with $l \leq f \leq u$. An $\epsilon$-bracket in $L_1$ is a bracket $[l, u]$ with $\mr{E}\vert u - l \vert < \epsilon$. The bracketing number $N_{[]}(\epsilon, \mc{F}, L_1)$ is the minimum number of $\epsilon$-brackets needed to cover $\mc{F}$. (The bracketing functions $l$ and $u$ must have finite $L_1$-norms but need not belong to $\mc{F}$).\\

\noindent The bracketing number is a quantity that measures the complexity of a function class. The lemma below provides an upper bound for a function class indexed by a finite dimensional bounded set. This result can be found in any empirical processes textbook such as \cite{VW96}. For completeness, we provide a proof.

\begin{Lemma}\label{brac_num_bound}
Let $\mc{F} = \{f_\theta \ \vert \ \theta \in \Theta \}$ be a collection of measurable functions indexed by a bounded subset $\Theta \subset \mathbb{R}^p$. Denote $D_\Theta = \sup\{\norm{\theta_1 - \theta_2} \ \vert \ \theta_1, \theta_2 \in \Theta\}$. Suppose that there exists a measurable function $g$ such that 
\begin{align}
\vert f_{\theta_1}(\xi) - f_{\theta_2}(\xi) \vert \leq g(\xi)\norm{\theta_1 - \theta_2} \label{brack_lip_cond}
\end{align} 
for every $\theta_1, \theta_2 \in \Theta$. If  $ \norm{g(\xi)}_1 \equiv \int \vert g(\xi)\vert dP < \infty$, then the bracketing numbers satisfy
\begin{align*}
N_{[]}(\epsilon\norm{g}_1, \mc{F}, L_1) \leq (\frac{\sqrt{p}D_\Theta}{\epsilon})^p
\end{align*}
for every $0 < \epsilon < D_\Theta$.
\end{Lemma}
\begin{proof}
To prove the result, we use brackets of the type $[f_\theta - \epsilon g / 2, f_\theta + \epsilon g / 2]$ for $\theta$ that ranging over a suitably chosen subset of $\Theta$ and these brackets have $L_1$-size $\epsilon\norm{g}_1$. If $\norm{\theta_1 - \theta_2} \leq \epsilon / 2$, then by the Lipschitz condition (\ref{brack_lip_cond}) we have $f_{\theta_1} - \epsilon g / 2 \leq f_{\theta_2} \leq f_{\theta_1} + \epsilon g / 2$.  Therefore, the brackets cover $\mc{F}$ if $\theta$ ranges over a grid of meshwidth $\epsilon / \sqrt{p}$ over $\Theta$. This grid has at most $(\sqrt{p} D_\Theta / \epsilon)^p$ grid points. Therefore the bracketing number $N_{[]}(\epsilon\norm{g}_1, \mc{F}, L_1)$ can be bounded by $(\sqrt{p} D_\Theta / \epsilon)^p$.
\end{proof}

\noindent\textbf{Remark:} The bracketing number has a very close relationship with the covering number, which is a better known quantity in machine learning. Let $N(\epsilon, \mc{F}, L_1)$ be the covering number of the set $\mc{F}$; that is, the minimal number of balls of $L_1$-radius $\epsilon$ needs to cover the set $\mc{F}$. Then the relation, $N(\epsilon, \mc{F}, L_1) \leq N_{[]}(2\epsilon, \mc{F}, L_1)$,  between covering number and bracketing number always holds. Moreover, this concept is also closely related to the VC-dimension. Usually, constructing and counting the number of brackets for a class of functions is easier to do than computing the minimum number of balls that covers the class.

\noindent Based on the bounds on the bracketing number for a function class with bounded index set, we can provide a concentration bound for $\sup_{\mb{x} \in \mc{P}} \vert F^\k(\mb{x}) - F(\mb{x}) \vert$.
\begin{Lemma}\label{concentration_bound}
For any $\delta > 0$ and $0 < \epsilon < \min\{D, \delta / (2L_F)\}$ we have
\begin{align*}
\mr{P}\{\sup_{\mb{x} \in \mc{P}} \vert F^\k(\mb{x}) - F(\mb{x}) \vert \geq \delta \} \leq 2 K_\mc{P}(\frac{D}{\epsilon})^p \exp\{-\frac{m^\k(\delta - 2L_F \epsilon)^2}{2(u_F - l_F)^2}\},
\end{align*}
where $L_F \equiv \min\{L_1, \ldots L_n\}$, $K_\mc{P} = (\sqrt{p})^p$, $u_F = \max\{\sup_{\mb{x} \in \mc{P}}f_i(\mb{x}) \ \vert \ i = 1, \ldots, n\}$ and $l_F = \min\{\inf_{\mb{x} \in \mc{P}}f_i(\mb{x}) \ \vert \ i = 1, \ldots, n\}$.
\end{Lemma}
\begin{proof}
Consider the function class $\mc{F} =\{f(\cdot, \mb{x}) \ \vert \ \mb{x} \in \mc{P}\}$ as defined in (\ref{stochastic_problem}), that is $f(i, \mb{x}) = f_i(\mb{x})$. Since $f_i(\cdot)$ each is assumed to be Lipschitz continuous with Lipschitz constant $L_i$, we must have $\vert f_i(\mb{x}) - f_i(\mb{y}) \vert \leq L_F\norm{\mb{x} - \mb{y}}$, where $L_F \equiv \max\{L_1, \ldots, L_n\}$. Moreover, the index set $\mc{P} \in \mr{R}^p$ for the function class $\mc{F}$ is assume to be bounded. Therefore all conditions for Lemma \ref{brac_num_bound} are satisfied and hence the number of brackets of the type $[f(\cdot, \mb{x}) - \epsilon L_F, f(\cdot, \mb{x}) + \epsilon L_F]$ satisfies
\begin{align*}
N_{[]}(\epsilon L_F, \mc{F}, L_1) \leq K_\mc{P}(\frac{D}{\epsilon})^p,
\end{align*}
for every $0 < \epsilon < D$, where $D = \sup\{\norm{\mb{x} - \mb{y}} \ \vert \ \mb{x}, \mb{y} \in \mc{P}\}$ and $K_\mc{P} = (\sqrt{p})^p$. Let $\Gamma \subset \mc{P}$ denote the set of indices of the centers of these brackets and $\xi_1, \ldots \xi_{m^\k}$ be the i.i.d. samples drawn at the $k$-th iteration of the algorithm. Since the brackets centered at $\Gamma$ cover $\mc{F}$, we must have
\begin{align*}
\sup_{\mb{x} \in \mc{P}}\vert \frac{1}{m^\k}\sum_{i=1}^{m^\k} f(\xi_i, \mb{x}) - \mr{E}f(\xi_i, \mb{x}) \vert \leq \max\{\vert \frac{1}{m^\k}\sum_{i=1}^{m^\k} f(\xi_i, \mb{y}) - \mr{E}f(\xi_i, \mb{y}) \vert \ \vert \ \mb{y} \in \Gamma \} + 2 \epsilon L_F.
\end{align*}
Consequently, for every $\delta \geq 0$ and $\epsilon < \min\{\delta / (2L_F), D\}$,  
\begin{align*}
\mr{P}\{\sup_{\mb{x} \in \mc{P}}\vert \frac{1}{m^\k}\sum_{i=1}^{m^\k} f(\xi_i, \mb{x}) - \mr{E}f(\xi_i, \mb{x}) \vert \geq \delta \} &\leq \mr{P}\{\max\{\vert \frac{1}{m^\k}\sum_{i=1}^{m^\k} f(\xi_i, \mb{y}) - \mr{E}f(\xi_i, \mb{y}) \vert \ \vert \ \mb{y} \in \Gamma \} + 2\epsilon L_F \geq \delta \} \\
&\leq \sum_{\mb{y} \in \Gamma} \mr{P}\{\vert \frac{1}{m^\k}\sum_{i=1}^{m^\k} f(\xi_i, \mb{y}) - \mr{E}f(\xi_1, \mb{y}) \vert \geq \delta - 2\epsilon L_F\} \tag{union bound} \\
&\leq \sum_{\mb{y} \in \Gamma} 2 \exp\{-\frac{2m^\k(\delta - 2L_F \epsilon)^2}{(u_F - l_F)^2}\} \tag{Hoeffding inequality} \\
&\leq  2 K_\mc{P}(\frac{D}{\epsilon})^p\exp\{-\frac{2m^\k(\delta - 2L_F \epsilon)^2}{(u_F - l_F)^2}\}.  \tag{$\vert \Gamma \vert \leq K_\mc{P}(\frac{D}{\epsilon})^p$}
\end{align*}
Since by definition, $F^\k(\mb{x}) = \frac{1}{m^\k}\sum_{i=1}^{m^\k} f(\xi_i, \mb{x})$ and $F(\mb{x}) = \mr{E}f(\xi_i, \mb{x})$, the desired result follows.
\end{proof}
\begin{Corollary}\label{concentraion_expectation}
When $m^\k \geq 3$, $$\mr{E}\sup_{\mb{x} \in \mc{P}} \vert F^\k(\mb{x}) - F(\mb{x}) \vert  \leq C_1\sqrt{\frac{\log m^\k}{m^\k}}$$ and 
$$\mr{E}\vert F^\k(\mb{x}^\k_*) - F(\mb{x}^*) \vert \leq C_1 \sqrt{\frac{\log m^\k}{m^\k}}$$ where
\begin{align*}
C_1 = 4(\vert u_F \vert + \vert l_F \vert) K_\mc{P} D^p\exp\{- p(\log\frac{u_F-l_F}{2\sqrt{2}L_F})\} + (u_F - l_F)\sqrt{p+1}.
\end{align*}
\end{Corollary}
\begin{proof}
First note that both $F^{(k)}(\cdot)$ and $F(\cdot)$ are bounded by $l_F$ and $u_F$; hence, $\sup_{\mb{x} \in \mc{P}}\vert F^\k(\mb{x}) - F(\mb{x})\vert \leq 2(\vert u_F \vert + \vert l_F \vert)$. Then for every $\delta \geq 0$, we have
\begin{align*}
\mr{E}\sup_{\mb{x} \in \mc{P}}\vert F^\k(\mb{x}) - F(\mb{x})\vert &\leq 2(\vert u_F \vert + \vert l_F \vert) \mr{P}\{\sup_{\mb{x} \in \mc{P}}\vert F^\k(\mb{x}) - F(\mb{x})\vert \geq \delta\} + \delta\ \mr{P}\{\sup_{\mb{x} \in \mc{P}}\vert F^\k(\mb{x}) - F(\mb{x})\vert < \delta\}\\
&\leq 4(\vert u_F \vert + \vert l_F \vert) K_\mc{P}(\frac{D}{\epsilon})^p\exp\{-\frac{2m^\k(\delta - 2L_F \epsilon)^2}{(u_F - l_F)^2}\} + \delta\\
&\leq 4(\vert u_F \vert + \vert l_F \vert) K_\mc{P} D^p\exp\{-\frac{2m^\k(\delta - 2L_F \epsilon)^2}{(u_F - l_F)^2} + p\log \frac{1}{\epsilon}\} + \delta.
\end{align*}
Now let $\delta = \frac{(u_F - l_F)\sqrt{4(p+1)\log\sqrt{m^\k}}}{\sqrt{m^\k}\sqrt{2}}$, $\epsilon = \frac{(u_F - l_F)}{2L_F\sqrt{m^\k}\sqrt{2}}$. Then
\begin{align*}
\mr{E}\sup_{\mb{x} \in \mc{P}}\vert F^\k(\mb{x}) - F(\mb{x})\vert &\leq 4(\vert u_F \vert + \vert l_F \vert) K_\mc{P} D^p\exp\{-(\sqrt{4(p+1)\log\sqrt{m^\k}} - 1)^2 - p(\log\frac{u_F-l_F}{2\sqrt{2}L_F}) + p\log\sqrt{m^\k}\}\\
&\quad + \frac{(u_F - l_F)\sqrt{4(p+1)\log\sqrt{m^\k}}}{\sqrt{m^\k}\sqrt{2}}.
\end{align*}
Note that $(x-1)^2 \geq x^2/4$ when $x \geq 2$. Thus, for $m^\k \geq 3$ and $p \geq 1$, $\sqrt{4(p+1)\log\sqrt{m^\k}} \geq 2$. Therefore
\begin{align*}
\mr{E}\sup_{\mb{x} \in \mc{P}}\vert F^\k(\mb{x}) - F(\mb{x})\vert &\leq 4(\vert u_F \vert + \vert l_F \vert) K_\mc{P} D^p\exp\{-(p+1)\log(\sqrt{m^\k}) + p\log\sqrt{m^\k} - p(\log\frac{u_F-l_F}{2\sqrt{2}L_F})\}\\
&\quad + \frac{(u_F - l_F)\sqrt{4(p+1)\log\sqrt{m^\k}}}{\sqrt{m^\k}\sqrt{2}}\\
&\leq C_1\sqrt{\frac{\log m^\k}{m^\k}},
\end{align*}
where $C_1 = 4(\vert u_F \vert + \vert l_F \vert) K_\mc{P} D^p\exp\{- p(\log\frac{u_F-l_F}{2\sqrt{2}L_F})\} + (u_F - l_F)\sqrt{p+1}$.\\
\noindent Next, we will obtain a bound for $\mr{E}\vert F^\k(\mb{x}^\k_*) - F(\mb{x}^*) \vert$. Lemma \ref{concentration_bound} implies both 
\begin{align}
F(\mb{x}^\k_*) - \delta \leq F^\k(\mb{x}^\k_*) \leq F(\mb{x}^\k_*) + \delta \label{base_event}
\end{align}
and 
\begin{align}
F(\mb{x}^*) - \delta \leq F^\k(\mb{x}^*) \leq F(\mb{x}^*) + \delta \label{base_event_2}
\end{align}
happen with probability at least $1- 2 K_\mc{P}(\frac{D}{\epsilon})^p \exp\{-\frac{m^\k(\delta - 2L_F \epsilon)^2}{2(u_F - l_F)^2}\}$. Consequently, on one hand
\begin{align*}
F^\k(\mb{x}^\k_*) &\geq F(\mb{x}^\k_*) - \delta \tag{by \ref{base_event}} \\
&\geq F(\mb{x}^*) -\delta \tag{optimality of $\mb{x}^*$ for $F(\cdot)$}
\end{align*}
On the other hand, 
\begin{align*}
F^\k(\mb{x}^\k_*) &\leq F^\k(\mb{x}^*) \tag{optimiality of $\mb{x}^\k_*$ for $F^\k(\cdot)$} \\
&\leq F(\mb{x}^*) + \delta \tag{by \ref{base_event_2}}
\end{align*}
Therefore, we have
\begin{align*}
\mr{P}\{\vert F^\k(\mb{x}^\k_*) - F(\mb{x}^*) \vert \geq \delta \} \leq 2 K_\mc{P}(\frac{D}{\epsilon})^p \exp\{-\frac{m^\k(\delta - 2L_F \epsilon)^2}{2(u_F - l_F)^2}\},
\end{align*}
and hence $\mr{E}\vert F^\k(\mb{x}^\k_*) - F(\mb{x}^*) \vert = C_1 \sqrt{\frac{\log m^\k}{m^\k}}$.
\end{proof}
\begin{Lemma}
Let $c_i \geq 0$ and $b_i \in \{0, 1\}$ for $i = 1, \ldots, n$. Assume that $\sum_{j=1}^n b_j = m < n$. Then for $0 < a < 1$ we have
\begin{align}\label{combinatorial_result}
\sum_{k=1}^n a^{\sum_{j=k}^n b_j}c_k \leq \sum_{k=1}^{m} a^{m-k+1}c_k + \sum_{k=m+1}^n c_k.
\end{align}
\end{Lemma}
\begin{proof}
The right hand side of \ref{combinatorial_result} is obtained by setting $b_i = 1$ for $i \leq m$ and $b_i = 0$ for $i > m$. We will show that this choice of $\{b_i\}$ maximizes $\sum_{k=1}^n a^{\sum_{j=k}^n b_j}c_k$. Consider an assignment of $b_i$ that there is a $b_r = 0$ for $r \leq m$ and $b_s = 1$ for $s > m$. Define a new assignment $b_i'$ such that there is $b_i' = b_i$ for $i \neq r, s$, $b_r' = 1$ and $b_s' = 0$. Then
\begin{align*}
\sum_{k=1}^n a^{\sum_{j=k}^n b_j}c_k &= \sum_{k = s+1}^n a^{\sum_{j=k}^n b_j}c_k + \sum_{k = r}^s a^{\sum_{j=k}^n b_j}c_k + \sum_{k=1}^{r - 1}a^{\sum_{j=k}^n b_j}c_k \\
&=\sum_{k = s+1}^n a^{\sum_{j=k}^n b_j'}c_k + \sum_{k = r + 1}^s a^{\sum_{j=k}^n b_j}c_k + \sum_{k=1}^{r}a^{\sum_{j=k}^n b_j'}c_k \\
&=\sum_{k = s+1}^n a^{\sum_{j=k}^n b_j'}c_k + a\sum_{k = r + 1}^s a^{\sum_{j=k}^n b_j'}c_k + \sum_{k=1}^{r}a^{\sum_{j=k}^n b_j'}c_k \\
&\leq \sum_{k = s+1}^n a^{\sum_{j=k}^n b_j'}c_k + \sum_{k = r + 1}^s a^{\sum_{j=k}^n b_j'}c_k + \sum_{k=1}^{r}a^{\sum_{j=k}^n b_j'}c_k \\
&=\sum_{k=1}^n a^{\sum_{j=k}^n b_j'}c_k.
\end{align*}
Therefore, such interchanges will always increase the value of $\sum_{k=1}^n a^{\sum_{j=k}^n b_j}c_k$ and hence setting $b_i = 1$ for $i \leq m$ and $b_i = 0$ for $i > m$ maximizes it.
\end{proof}
\noindent With the developments of the above lemmas and corollary we are ready to state and prove the main results.
\begin{Theorem}\label{asfw_theorem}
Let $\{\mathbf{x}^{(k)}\}_{k \geq 1}$ be the sequence generated by Algorithm \ref{cond_grad_1} for solving Problem $(\ref{general_problem})$, $N$ be the number of vertices used to represent $\mb{x}^{(k)}$ (if VRU is implemented by using Carath\'eodory's theorem, $ N = p + 1$, otherwise $N = \vert V \vert$) and $F^*$ be the optimal value of the problem. Let $\rho = \min\{\frac{1}{2}, \frac{\Omega_{\mathcal{P}}^2\sigma_F}{16N^2 L_F D^2}\}$ where $\sigma_F = \min\{\sigma_1, \ldots, \sigma_n\}$, $L_F = \max\{L_1, \ldots, L_n\}$. Set $m^{(i)} = \lceil 1 / (1-\rho)^{2i + 2}\rceil$. Then for every $k \geq 1$
\begin{align}\label{rate_of_convergence}
\mathbb{E}\{F(\mathbf{x}^{(k+1)}) - F^*\} \leq C_2(1 - \beta)^{(k-1)/2},
\end{align}
where $C_2$ is a deterministic constant and $0 < \beta < \rho \leq 1/2$.
\end{Theorem}
\begin{proof}
At iteration $k$, let $\mb{x}^\k$ denote the current solution, $\xi_1, \ldots, \xi_{m^\k}$ denote the samples obtained in the algorithm, $\mb{d}^\k$ denote the direction that the algorithm will take at this step and $\gamma^\k$ denote the step length. Define
$F^\k(\mb{x}) = \frac{1}{m^\k}\sum_{i=1}^{m^\k} f(\xi_i, \mb{x})$, $\mb{x}^\k_* = \arg\min_{\mb{x} \in \mc{P}}F^\k(\mb{x})$ and $F^\k_* = F^\k(\mb{x}^\k_*)$. Note that $F^\k$ is Lipschitz continuous with Lipschitz constant $L^\k = \frac{1}{m^\k}\sum_{i=1}^{m^\k}L_{\xi_i}$ and strongly convex with constant $\sigma^\k = \frac{1}{m^\k}\sum_{i=1}^{m^\k}\sigma_{\xi_i}$. In addition, the stochastic gradient $\mb{g}^\k = \nabla F^\k(\mb{x})$. From the choice of $\mb{d}^{(k)}$ in the algorithm, 
\begin{align*}
\langle \mb{g}^{(k)}, \mb{d}^{(k)} \rangle \leq \frac{1}{2}(\langle \mb{g}^{(k)}, \mb{p}^\k - \mb{x}^\k \rangle + \langle \mb{g}^{(k)}, \mb{x}^\k - \mb{u}^\k \rangle) = \frac{1}{2}\langle \mb{g}^{(k)}, \mb{p}^{(k)} - \mb{u}^{(k)} \rangle \leq 0.
\end{align*}
Hence, we can lower bound $\langle \mb{g}^{(k)}, \mb{d}^{(k)} \rangle^2$ by 
\begin{align*}
\langle \mb{g}^{(k)}, \mb{d}^{(k)} \rangle^2 &\geq \frac{1}{4} \langle \mb{g}^{(k)}, \mb{u}^{(k)} - \mb{p}^\k \rangle^2 \\
&\geq  \frac{1}{4}\max_{\mb{p} \in V, \mb{u} \in U^\k}\langle \mb{g}^{(k)}, \mb{u} - \mb{p} \rangle^2 \tag{definition of $\mb{p}^\k$ and $\mb{u}^\k$} \\
& = \frac{1}{4} \max_{\mb{p} \in V, \mb{u} \in U^\k}\langle \nabla F^\k(\mb{x}^\k), \mb{u} - \mb{p} \rangle^2 \tag{$\mb{g}^\k = \nabla F^\k(\mb{x}^\k)$} \\
& \geq \frac{1}{4} \frac{\Omega_{\mathcal{P}}^2}{\vert U^{(k)} \vert^2} \frac{\langle \nabla F^\k(\mathbf{x}^\k), \mathbf{x}^\k - \mathbf{x}_*^\k \rangle^2}{\norm{\mathbf{x}^\k - \mb{x}_*^\k}^2} \tag{by Lemma \ref{linear_approx_lower_bound}} \\
&\geq \frac{\Omega_{\mathcal{P}}^2}{4N^2}\frac{\{F^\k(\mb{x}^\k) - F_*^\k\}^2}{\norm{\mathbf{x}^\k - \mb{x}_*^\k}^2} \tag{Convexity of $F^\k(\cdot)$} \\
& \geq \frac{\Omega_{\mathcal{P}}^2 \sigma^\k}{8N^2} \{F^\k(\mb{x}^\k) - F^\k_*\} \tag{by strong convexity of $F^\k(\cdot)$} \\
& \geq \frac{\Omega_{\mathcal{P}}^2 \sigma_F}{8N^2} \{F^\k(\mb{x}^\k) - F^\k_*\}.
\end{align*}
Similarly, we can upper bound $\langle \mb{g}^{(k)}, \mb{d}^{(k)} \rangle$ by 
\begin{align*}
\langle \mb{g}^{(k)}, \mb{d}^{(k)} \rangle & \leq \frac{1}{2}\langle \mb{g}^{(k)}, \mb{p}^{(k)} - \mb{u}^{(k)} \rangle \\
& \leq  \frac{1}{2} \langle \mb{g}^{(k)}, \mb{x}_*^\k - \mb{x}^{(k)}\rangle \tag{definition of $\mb{p}^\k$ and $\mb{u}^\k$} \\
& = \frac{1}{2}\langle \nabla F^\k(\mb{x}^\k), \mb{x}_*^\k - \mb{x}^\k \rangle \tag{$\mb{g}^\k = \nabla F^\k(\mb{x}^\k)$}\\
& \leq \frac{1}{2}\{F^\k_* - F^\k(\mb{x}^\k)\} \tag{Convexity of $F(\cdot)$}.
\end{align*}
With the above bounds, we can separate our analysis into the following four cases at iteration $k$
\begin{description}
\item[($A^{(k)}$)]\hspace{0.5cm} $\gamma_{\max}^\k \geq 1$ and $\gamma^\k \leq 1$ .
\item[($B^{(k)}$)]\hspace{0.5cm} $\gamma_{\max}^\k \geq 1$ and $\gamma^\k \geq 1$.
\item[($C^{(k)}$)]\hspace{0.5cm} $\gamma_{\max}^\k < 1$ and $\gamma^\k < \gamma_{\max}^\k$.
\item[($D^{(k)}$)]\hspace{0.5cm} $\gamma_{\max}^\k < 1$ and $\gamma^\k = \gamma_{\max}^\k$.
\end{description}
By the descent lemma, we have 
\begin{align}
F^\k(\mathbf{x}^{(k + 1)}) = F^\k(\mathbf{x}^{(k)} + \gamma^{(k)}\mathbf{d}^{(k)}) &\leq F^\k(\mathbf{x}^{(k)}) + \gamma^{(k)}\langle \nabla F^\k(\mathbf{x^{(k)}}), \mathbf{d}^{(k)} \rangle + \frac{L^\k(\gamma^{(k)})^2}{2}\norm{\mathbf{d}^{(k)}}^2 \nonumber \\
&= F^\k(\mathbf{x}^{(k)}) + \gamma^{(k)}\langle \mb{g}^{(k)}, \mathbf{d}^{(k)} \rangle + \frac{L^\k(\gamma^{(k)})^2}{2}\norm{\mathbf{d}^{(k)}}^2. \label{basic_neq}
\end{align}

\noindent In case $(A^{(k)})$, let $\delta_{A^{(k)}}$ denote the indicator function for this case. Then
\begin{align*}
\delta_{A^{(k)}}\{F^\k(\mathbf{x}^{(k + 1)}) - F_*^\k\} &\leq \delta_{A^{(k)}} \{ F^\k(\mathbf{x}^{(k)}) - F_*^\k + \gamma^{(k)}\langle \mb{g}^{(k)}, \mathbf{d}^{(k)} \rangle + \frac{L^\k(\gamma^{(k)})^2}{2}\norm{\mathbf{d}^{(k)}}^2 \} \\
&= \delta_{A^{(k)}}\{F^\k(\mathbf{x}^{(k)}) - F_*^\k - \frac{\langle \mb{g}^\k, \mb{d}^\k \rangle^2}{2L^\k\norm{\mb{d}^\k}^2} \}\tag{definition of $\gamma^\k$ in case $A^{(k)}$}  \\
&\leq \delta_{A^\k}\{(1- \frac{\Omega_{\mathcal{P}}^2\sigma_F}{16N^2 L^\k D^2})(F^\k(\mathbf{x}^{(k)}) - F_*^\k)\}\\
&\leq \delta_{A^\k}\{(1- \frac{\Omega_{\mathcal{P}}^2\sigma_F}{16N^2 L_F D^2})(F^\k(\mathbf{x}^{(k)}) - F_*^\k)\}
\end{align*}

\noindent In case $(B^{(k)})$, since $\gamma^\k > 1$, we have 
\begin{align}
&-\langle \mb{g}^\k, \mb{d}^\k \rangle > L^\k \norm{\mb{d}^\k}^2 \quad \quad \quad \text{and} \label{case_b_in_prod_bound}\\
&\gamma^{(k)}\langle \mb{g}^{(k)}, \mathbf{d}^{(k)} \rangle + \frac{L^\k(\gamma^{(k)})^2}{2}\norm{\mathbf{d}^{(k)}}^2 \leq \langle \mb{g}^{(k)}, \mathbf{d}^{(k)} \rangle + \frac{L^\k}{2}\norm{\mathbf{d}^{(k)}}^2. \label{case_b_monotone}
\end{align}
Use $\delta_{B^{(k)}}$ to denote the indicator function for this case. Then,
\begin{align*}
\delta_{B^{(k)}} \{F^\k(\mathbf{x}^{(k + 1)}) - F^\k_* \}&\leq \delta_{B^{(k)}} \{F^\k(\mathbf{x}^{(k)}) - F_*^\k + \gamma^{(k)}\langle \nabla F^\k(\mathbf{x}^{(k)}), \mathbf{d}^{(k)} \rangle + \frac{L^\k(\gamma^{(k)})^2}{2}\norm{\mathbf{d}^{(k)}}^2\}  \\
&= \delta_{B^{(k)}} \{F^\k(\mathbf{x}^{(k)}) - F_*^\k + \gamma^{(k)}\langle \mb{g}^{(k)}, \mathbf{d}^{(k)} \rangle + \frac{L^\k(\gamma^{(k)})^2}{2}\norm{\mathbf{d}^{(k)}}^2\\
&\leq \delta_{B^{(k)}} \{F^\k(\mathbf{x}^{(k)}) - F_*^\k +\langle \mb{g}^{(k)}, \mathbf{d}^{(k)} \rangle + \frac{L^\k}{2}\norm{\mathbf{d}^{(k)}}^2 \} \tag{ by (\ref{case_b_monotone})}\\
&\leq \delta_{B^{(k)}} \{ F^\k(\mathbf{x}^{(k)}) - F_*^\k + \frac{1}{2}\langle \mb{g}^{(k)}, \mathbf{d}^{(k)} \rangle \} \tag{by (\ref{case_b_in_prod_bound})} \\
&\leq \delta_{B^{(k)}} \{\frac{1}{2}(F^\k(\mathbf{x}^{(k)}) - F_*^\k)\}
\end{align*}

In case $(C^{(k)})$, let $\delta_{C^{(k)}}$ be the indicator function for this case and we can use exactly the same argument as in case (A) to obtain the following inequality 
\begin{align*}
\delta_{C^{(k)}} \{F^\k(\mathbf{x}^{(k + 1)}) - F^\k_*\}  &\leq \delta_{C^{(k)}} \{F^\k(\mathbf{x}^{(k)}) - F_*^\k - \frac{\langle \mb{g}^\k, \mb{d}^\k \rangle^2}{2L^\k\norm{\mb{d}^\k}^2} \}\\
&\leq \delta_{C^\k}\{(1- \frac{\Omega_{\mathcal{P}}^2\sigma_F}{16N^2 L_F D^2})(F^\k(\mathbf{x}^{(k)}) - F_*^\k)\}
\end{align*}

\noindent Case $(D^{(k)})$ is the so called ``drop step" in the conditional gradient algorithm with away-steps. Use $\delta_{D^{(k)}}$ to denote the indicator function for this case. Note that $\gamma^\k = \gamma_{\max}^\k \leq -\langle \mb{g}^\k, \mb{d}^\k \rangle / (L^\k \norm{\mb{d}^\k}^2)$ in this case. Hence, we have
\begin{align*}
\delta_{D^{(k)}} \{(F^\k(\mathbf{x}^{(k + 1)}) - F_*^\k) \} &\leq \delta_{D^{(k)}}\{ F^\k(\mathbf{x}^{(k)}) - F_*^\k + \gamma^{(k)}\langle \nabla F^\k(\mathbf{x^{(k)}}), \mathbf{d}^{(k)} \rangle + \frac{L^\k(\gamma^{(k)})^2}{2}\norm{\mathbf{d^{(k)}}}^2 \} \\
&= \delta_{D^{(k)}}\{ F^\k(\mathbf{x}^{(k)}) - F_*^\k + \gamma^{(k)}\langle \mb{g}^{(k)}, \mathbf{d}^{(k)} \rangle + \frac{L^\k(\gamma^{(k)})^2}{2}\norm{\mathbf{d^{(k)}}}^2 \}\\
&\leq \delta_{D^{(k)}} \{ F^\k(\mathbf{x}^{(k)}) - F_*^\k + \frac{\gamma^\k}{2}\langle \mb{g}^{(k)}, \mathbf{d}^{(k)} \rangle\} \\
&\leq \delta_{D^{(k)}} \{ F^\k(\mathbf{x}^{(k)}) - F_*^\k\}.
\end{align*}
\noindent Define $\rho = \min\{\frac{1}{2}, \frac{\Omega_{\mathcal{P}}^2\sigma_F}{16N^2 L_F D^2}\}$. Note that $\rho$ is a deterministic constant between 0 and 1. Therefore we have
\begin{align*}
F^\k(\mb{x}^{(k+1)}) - F^\k_* &\leq (\{1 - \rho)^{\{1 - \delta_{D^\k}\}} (F^\k(\mb{x}^\k) - F^\k_*) \\
&= (1 - \rho)^{\{1 - \delta_{D^\k}\}}(F^{(k-1)}(\mb{x}^\k) - F^{(k-1)}_*) \\
&\quad + (1 - \rho)^{\{1 - \delta_{D^\k}\}}\{F^\k(\mb{x}^\k) - F^\k_* -F^{(k-1)}(\mb{x}^\k) + F^{(k-1)}_*\}\\
&= (1 - \rho)^{\{1 - \delta_{D^\k}\}}(F^{(k-1)}(\mb{x}^\k) - F^{(k-1)}_*) \\
&\quad + (1 - \rho)^{\{1 - \delta_{D^\k}\}}\{F^\k(\mb{x}^\k) - F(\mb{x}^\k)+ F(\mb{x}^\k) -F^{(k-1)}(\mb{x}^\k) + F^*  - F^\k_*  + F^{(k-1)}_* - F^*\}\\
&\leq (1 - \rho)^{\{1 - \delta_{D^\k}\}}(F^{(k-1)}(\mb{x}^\k) - F^{(k-1)}_*) \\
&\quad + (1 - \rho)^{\{1 - \delta_{D^\k}\}}\{\vert F^\k(\mb{x}^\k) - F(\mb{x}^\k) \vert + \vert F^{(k-1)}(\mb{x}^\k) - F(\mb{x}^\k) \vert + \vert F^\k_* - F^*\vert  + \vert F^{(k-1)}_* - F^* \vert \}\\
&\leq (1 - \rho)^{\sum_{i=1}^k\{1 - \delta_{D^{(i)}}\}}(F^{(0)}(\mb{x}^{(1)}) - F_*^{(0)}) + \\
&\quad \sum_{i=1}^k (1 - \rho)^{\sum_{j=i}^k\{1 - \delta_{D^{(j)}}\}}\{\vert F^{(i)}(\mb{x}^{(i)}) - F(\mb{x}^{(i)}) \vert + \vert F^{(i-1)}(\mb{x}^{(i)}) - F(\mb{x}^{(i)}) \vert + \\
&\quad \vert F^{(i)}_* - F^*\vert  + \vert F^{(i-1)}_* - F^* \vert\}.
\end{align*}
\noindent At iteration $k$, there are at most $(k+1)/2$ drop steps, i.e., at most $(k+1)/2$ $\delta_{D^{(i)}}$'s equal to 1. Then by Lemma \ref{combinatorial_result}, we have 
\begin{align*}
&\quad\sum_{i=1}^k (1 - \rho)^{\sum_{j=i}^k\{1 - \delta_{D^{(j)}}\}}\{\vert F^{(i)}(\mb{x}^{(i)}) - F(\mb{x}^{(i)}) \vert + \vert F^{(i-1)}(\mb{x}^{(i)}) - F(\mb{x}^{(i)}) \vert + \vert F^{(i)}_* - F^*\vert  + \vert F^{(i-1)}_* - F^* \vert\}\\
&\leq \sum_{i=k/2}^k\{\vert F^{(i)}(\mb{x}^{(i)}) - F(\mb{x}^{(i)}) \vert + \vert F^{(i-1)}(\mb{x}^{(i)}) - F(\mb{x}^{(i)}) \vert +
\vert F^{(i)}_* - F^*\vert  + \vert F^{(i-1)}_* - F^* \vert\} \\
&+ \sum_{i = 1}^{k/2 - 1}(1 - \rho)^{k/2 - i}\{\vert F^{(i)}(\mb{x}^{(i)}) - F(\mb{x}^{(i)}) \vert + \vert F^{(i-1)}(\mb{x}^{(i)}) - F(\mb{x}^{(i)}) \vert +\vert F^{(i)}_* - F^*\vert  + \vert F^{(i-1)}_* - F^* \vert\}.
\end{align*}
Therefore
\begin{align*}
F^\k(\mb{x}^{(k+1)}) - F^\k_* &\leq (1 - \rho)^{\frac{k-1}{2}}(u_F-l_F) \\
&+ \sum_{i=k/2}^k\{\vert F^{(i)}(\mb{x}^{(i)}) - F(\mb{x}^{(i)}) \vert + \vert F^{(i-1)}(\mb{x}^{(i)}) - F(\mb{x}^{(i)}) \vert +
\vert F^{(i)}_* - F^*\vert  + \vert F^{(i-1)}_* - F^* \vert\} \\
&+ \sum_{i = 1}^{k/2 - 1}(1 - \rho)^{k/2 - i}\{\vert F^{(i)}(\mb{x}^{(i)}) - F(\mb{x}^{(i)}) \vert + \vert F^{(i-1)}(\mb{x}^{(i)}) - F(\mb{x}^{(i)}) \vert +\vert F^{(i)}_* - F^*\vert  + \vert F^{(i-1)}_* - F^* \vert\}.
\end{align*}
In addition, $F^\k(\mb{x}^{(k+1)}) - F^\k_* = F(\mb{x}^{(k+1)}) - F^* + (F^\k(\mb{x}^{(k+1)}) - F(\mb{x}^{(k+1)})) + (F^* - F^\k_*)$. Thus
\begin{align*}
F(\mb{x}^{(k+1)}) - F^* &\leq (1 - \rho)^{\frac{k-1}{2}}(u_F-l_F) \\
&+ \sum_{i=k/2}^{k+1}\{\vert F^{(i)}(\mb{x}^{(i)}) - F(\mb{x}^{(i)}) \vert + \vert F^{(i-1)}(\mb{x}^{(i)}) - F(\mb{x}^{(i)}) \vert +
\vert F^{(i)}_* - F^*\vert  + \vert F^{(i-1)}_* - F^* \vert\} \\
&+ \sum_{i = 1}^{k/2 - 1}(1 - \rho)^{k/2 - i}\{\vert F^{(i)}(\mb{x}^{(i)}) - F(\mb{x}^{(i)}) \vert + \vert F^{(i-1)}(\mb{x}^{(i)}) - F(\mb{x}^{(i)}) \vert +\vert F^{(i)}_* - F^*\vert  + \vert F^{(i-1)}_* - F^* \vert\}.
\end{align*}
Note that for any deterministic $\mb{x} \in \mc{P}$, we have $\mr{E} F^{\k}(\mb{x}) = F(\mb{x})$. In addition, by Corollary \ref{concentraion_expectation}, the following bound holds for every iteration $k$
\begin{align*}
\mr{E}\vert F^\k(\mb{x}^\k) - F(\mb{x}^\k) \vert \leq \mr{E}\sup_{\mb{x} \in \mc{P}}\vert F^\k(\mb{x}) - F(\mb{x})\vert \leq C_1 \sqrt{\frac{\log m^\k}{m^\k}}
\end{align*}
and
\begin{align*}
\mr{E}\vert F^\k_* - F^*\vert \leq C_1 \sqrt{\frac{\log m^\k}{m^\k}}.
\end{align*}

\noindent Combining all above bounds and use $m^{(i)} = \lceil 1 / (1-\rho)^{2i + 2}\rceil$, we have
\begin{align*}
\mr{E}\{F(\mb{x}^{(k+1)}) - F^*\} &\leq (1 - \rho)^{\frac{k-1}{2}}(u_F-l_F) \\
&\quad + 2C_1\{\sum_{i=k/2}^{k+1}(\sqrt{\frac{\log m^{(i)}}{m^{(i)}}} + \sqrt{\frac{\log m^{(i-1)}}{m^{(i-1)}}}) + \sum_{i=1}^{k/2 - 1}(1 - \rho)^{k/2-i}(\sqrt{\frac{\log m^{(i)}}{m^{(i)}}} + \sqrt{\frac{\log m^{(i-1)}}{m^{(i-1)}}})\}\\
&\leq (1 - \rho)^{\frac{k-1}{2}}(u_F-l_F) + 4C_1\{\sum_{i=k/2}^{k+1}\sqrt{\frac{\log m^{(i-1)}}{m^{(i-1)}}} + \sum_{i=1}^{k/2 - 1}(1-\rho)^{k/2 -i}\sqrt{\frac{\log m^{(i-1)}}{m^{(i-1)}}}\} \tag{$\frac{\log x}{x}$ decreases for $x > e$}\\
&\leq (1 - \rho)^{\frac{k-1}{2}}(u_F-l_F) + 4C_1\sqrt{2\log\frac{1}{1-\rho}}\{\sum_{i= k/2}^{k+1}(1-\rho)^i\sqrt{i} + \sum_{i=1}^{k/2 - 1}(1-\rho)^{k/2}\sqrt{i}\} \\
&\leq C_2 (1 - \beta)^{\frac{k-1}{2}}
\end{align*}
for some constant $C_2$ and $ 0 < \beta < \rho < 1$.
\end{proof}

\noindent\textbf{Remark:} The proof of Theorem \ref{asfw_theorem} doesn't have any stochastic arguments involved until the very end and we use Lemma \ref{combinatorial_result} to get rid of the indicator function for the `drop-steps' so that the stochastic arguments based on concentration inequalities can be applied. Note that we cannot take expectation on the stochastic gradients and utilize their unbiasedness property because of the presence of the indicator functions. This proof technique is specifically designed for the `drop-step' in ASFW and can be useful in analyzing other similar algorithms.

\begin{Corollary}\label{as_conv}
Let $\{\mathbf{x}^{(k)}\}_{k \geq 1}$ be the sequence generated by Algorithm \ref{cond_grad_1} for solving Problem $(\ref{general_problem})$. Then $$\frac{F(\mathbf{x}^\k) - F^*}{(1 - \omega)^{\frac{k-1}{2}}} \rightarrow 0$$ almost surely as $k$ tends to infinity for some $0 < \omega < \beta$. Therefore $F(\mathbf{x}^\k)$ linearly converges to $F^*$ almost surely.
\end{Corollary}
\begin{proof}
For every $\epsilon > 0$, let $E^\k$ denotes the event that $ (F(\mathbf{x}^\k) - F^*) / (1 - \omega)^{(k-1)/2} > \epsilon$. By Markov inequality
\begin{align*}
\sum_{k=2}^\infty \mr{P}(E^\k) &= \sum_{k=1}^\infty \mr{P}((F(x^\k) - F^*) / (1 - \omega)^{(k-1)/2} > \epsilon) \\
&\leq \sum_{k=2}^\infty \frac{\mr{E}\{F(x^\k) - F^*\}}{\epsilon(1 - \omega)^{(k-1)/2}} \\
&\leq \frac{C_2}{\epsilon}\sum_{k=2}^\infty (\frac{1 - \beta}{1 - \omega})^{\frac{k-1}{2}} \\
&< \infty.
\end{align*}
Therefore $\sum_{k=2}^\infty \mr{P}(E^\k) < \infty$ and the Borel-Cantelli lemma implies that $\mr{P}(\lim\sup_{k\rightarrow\inf}E^\k) = 0$ which implies $(F(\mb{x}^\k) - F^*) / (1 - \omega)^{(k-1)/2}$ converges to $0$ almost surely. This implies that every sequence generates by Algorithm \ref{cond_grad_1} linearly converges to the optimal function value almost surely.
\end{proof}

\noindent\textbf{Remark:} Note that the result in Corollary \ref{as_conv} only relies on the property that an algorithm converges linearly in expectation. Therefore, we can apply exactly the same argument to show that every sequence generated by the algorithm in \cite{JZ13} converges linearly almost surely.
\begin{Corollary}\label{asfw_num_stoc_grad}
To obtain an $\epsilon$-accurate solution, Algorithm \ref{cond_grad_1} requires $O((1 / \epsilon)^{4\eta})$ of stochastic gradient evaluations, where $0 < \eta = \log(1 - \rho) / \log(1- \beta) < 1$.
\end{Corollary}
\begin{proof}
Let $k$ be the total number of iterations performed by Algorithm \ref{cond_grad_1} so that an $\epsilon$-accurate solution is obtained for the first time. Theorem \ref{asfw_theorem} implies $ C_2(1 - \beta)^\frac{k-1}{2} < \epsilon $ and hence $k \geq 1 + 2 \log\epsilon / \log(1 - \beta)$. In iteration $i$ of Algorithm \ref{cond_grad_1}, $m^{(i)} = 1 / (1 - \rho)^{2i + 2}$ of stochastic gradient evaluations are performed. Thus, the total number of stochastic gradient evaluations until iteration $k$ is
\begin{align*}
\sum_{i = 1}^k m^{(i)} &= \sum_{i=1}^k \frac{1}{(1 - \rho)^{(2i+2)}} \\
& = \frac{1}{(1 - \rho)^2}\frac{1 / (1 - \rho)^2 - 1 /(1 - \rho)^{2k + 2}}{1 - 1 / (1 - \rho)^2} \\
&\leq \frac{2}{(1 - \rho)^{2k + 4}} \leq \frac{2}{(1 - \rho)^4}\exp\{-2k\log(1 - \rho)\} \\
&\leq \frac{2}{(1 - \rho)^4}\exp\{-2\log(1 - \rho) - 4\frac{\log \epsilon \log(1 - \rho)}{\log(1 - \beta)}\}\\
&= O((\frac{1}{\epsilon})^ \frac{4\log(1 - \rho)}{\log(1 - \beta)})\\
&= O((\frac{1}{\epsilon})^{4\eta}).
\end{align*}
\end{proof}
\begin{Theorem}
Let $\{\mathbf{x}^{(k)}\}_{k \geq 1}$ be the sequence generated by Algorithm \ref{cond_grad_pair} for solving Problem $(\ref{general_problem})$, $N$ be the number of vertices used to represent $\mb{x}^{(k)}$ (if VRU is implemented by using Carath\'eodory's theorem, $ N = p + 1$, otherwise $N = \vert V \vert$) and $F^*$ be the optimal value of the problem. Let $\kappa = \min\{\frac{1}{2}, \frac{\Omega_{\mathcal{P}}^2\sigma_F}{8N^2 L_F D^2}\}$ where $\sigma_F = \min\{\sigma_1, \ldots, \sigma_n\}$, $L_F = \max\{L_1, \ldots, L_n\}$. Set $m^{(i)} = \lceil 1 / (1-\kappa)^{2i + 2}\rceil$. Then for every $k \geq 1$
\begin{align}\label{rate_of_convergence}
\mathbb{E}\{F(\mathbf{x}^{(k+1)}) - F^*\} \leq C_3(1 - \phi)^{k / (3 \vert V \vert ! + 1)}
\end{align}
where $C_3$ is a deterministic constant and $0 < \phi < \kappa \leq  1/2$.
\end{Theorem}
\begin{proof}
Since $\mb{d}^\k = \mb{p}^\k - \mb{u}^\k$, similar to the proof of Theorem \ref{asfw_theorem}, we have
\begin{align*}
\langle \mb{g}^\k, \mb{d}^\k \rangle^2 &\geq \frac{\Omega_\mc{P}^2\sigma_F}{4N^2}\{F^\k(\mb{x}^\k) - F^\k_*\} \\
\langle \mb{g}^\k, \mb{d}^\k \rangle &\leq  \frac{1}{2}(F^\k_* - F^\k(\mb{x}^\k)).
\end{align*}
The remaining proof for Theorem \ref{asfw_theorem} could also apply here except that the case $D^{(k)}$ can be either a `drop step' or a so-called `swap step'. A swap step moves the weight of a active vertex to another active vertex. There are at most $(1 - \frac{1}{3\vert V \vert! + 1})k$ drop steps and swap steps after $k$ iteration. The same argument as in Theorem \ref{asfw_theorem} implies 
\begin{align*}
\mathbb{E}\{F(\mathbf{x}^{(k+1)}) - F^*\} \leq C_3(1 - \phi)^{k / (3 \vert V \vert ! + 1)}
\end{align*}
for a deterministic constant $C_3$ and $0 < \phi < \kappa \leq 1/2$.
\end{proof}
\begin{Corollary}
Let $\{\mathbf{x}^{(k)}\}_{k \geq 1}$ be the sequence generated by Algorithm \ref{cond_grad_pair} for solving Problem $(\ref{general_problem})$. Then $$\frac{F(\mathbf{x}^\k) - F^*}{(1 - \psi)^{\frac{k}{3\vert V \vert ! + 1}}} \rightarrow 0$$ almost surely as $k$ tends to infinity for some $0 < \psi < \phi$. Therefore $F(\mathbf{x}^\k)$ linearly converges to $F^*$ almost surely.
\end{Corollary}
\noindent Proof of this Corollary is almost the same as the proof of Corollary \ref{as_conv}.
\begin{Corollary}
To obtain an $\epsilon$-accurate solution, Algorithm \ref{cond_grad_pair} requires $O((1 / \epsilon)^{(6\vert V \vert! + 2)\xi})$ of stochastic gradient evaluations, where $0 < \zeta = \log(1 - \rho) / \log(1- \phi) < 1$.
\end{Corollary}
\noindent Proof of this Corollary is the same as the proof of Corollary \ref{asfw_num_stoc_grad}.
\section{Numerical Experiments}
\subsection{Simulated Data}
We apply the proposed algorithms to the synthetic problem below:
\begin{align*}
&\text{minimize} \quad \norm{\mathbf{Ax} - \mathbf{b}}_2^2 + \frac{1}{2}\norm{\mathbf{x}}_2^2 \\
&\text{such that} \quad l \leq x_1 \leq x_2 \leq \cdots \leq x_p \leq u,
\end{align*}
where $\mathbf{A} \in \mathbb{R}^{n \times p}$, $\mathbf{b} \in \mathbb{R}^n$ and $\mathbf{x} \in \mathbb{R}^p$. We generated the entries of $\mathbf{A}$ and $\mathbf{b}$ from  standard normal distribution and set $n = 10^6$, $p = 1000$, $l = -1$ and $u = 1$. This problem can be viewed as minimizing a sum of strongly convex functions subject to a polytope constraint. Such problems can be found in shape restricted regression literatures.  We compared the ASFW and PSFW with two variance-reduced stochastic methods, the variance-reduced stochastic Frank-Wolfe (SVRF) method \citep{LH16} and the proximal variance-reduced stochastic gradient (Prox-SVRG) method \citep{JZ13} \citep{LZ14}. Both Prox-SVRG and SVRF are epoch based algorithms. They first fix a reference point and compute the exact gradient at the reference point at the beginning of each epoch. Within each epoch, both algorithms compute variance reduced gradients in every step using the control variates technique based on the reference point. The major difference between them is that in every iteration, the Prox-SVRG takes a proximal gradient step and the SVRF takes a Frank-Wolfe step. For detailed implementations of SVRF, we followed Algorithm 1 in \cite{LH16} and chose the parameters according to Theorem 1 in \cite{LH16}. For the Prox-SVRG, we followed the Algorithm in \cite{LZ14} and set the number of iterations in each epoch to be $m = 2n$ and set the step size to be $\gamma = 0.1 / L$ found by \cite{LZ14} to give the best results for Prox-SVRG, where $n$ is the sample size and $L$ is the Lipschitz constant of the gradient of the objective function. For ASFW and PSFW implementations, we followed Algorithm \ref{cond_grad_1} and Algorithm \ref{cond_grad_pair} and used adaptive step sizes since we know the Lipschitz constants of the gradients of the objective functions. The number of samples that we used to compute stochastic gradients for ASFW and PSFW was set to be $1.04^k + 100$ at the iteration $k$. The linear optimization sub-problems in Frank-Wolfe algorithms and the projection step in Prox-SVRG were solved by using the GUROBI solver. We summarize the parameters that were used in the algorithms at iteration $k$ and epoch $t$ in Table 2.\\
\begin{table}[h]
\label{alg_para}
\begin{center}
\begin{tabular}{|c|ccc|}
\hline
	& step-size & batch-size & \#iterations \\
\hline
ASFW &$\min\{-\langle \mb{g}^\k, \mb{d}^\k \rangle / (L^\k \norm{\mb{d}^\k}^2), \gamma_{\max}\}$ & $100 + 1.04^k$ & N/A \\
PSFW &$\min\{-\langle \mb{g}^\k, \mb{d}^\k \rangle / (L^\k \norm{\mb{d}^\k}^2), \gamma_{\max}\}$ & $100 + 1.04^k$ & N/A \\
SVRF &$2 / (k + 1)$ & $96(k + 1)$ & $2^{t + 3} - 2$ \\
SVRG &$0.1 / L$ &$1$ & $2n$ \\
\hline
\end{tabular}
\caption{In ASFW and PSFW, $\mb{g}^\k$ is the stochastic gradient, $L^\k$ is the Lipschitz constant of the stochastic gradient at iteration $k$, $\mb{d}^\k$ is the direction the algorithms take at iteration $k$ and $\gamma_{\max}$ is the maximum of the possible step sizes (see Algorithm \ref{cond_grad_1} and \ref{cond_grad_pair}). In Prox-SVRG, $L$ is the Lipschitz constant of the gradient of the objection function and $n$ is the sample size.}
\end{center}
\end{table}

\noindent To make fair comparisons, we use the same starting point for all four algorithms. The loss functions using ASFW, PSFW and Prox-SVRG and the running minimum using SVRF are plotted against CPU time. From the plot, we can see that ASFW and PSFW performed as well as or slightly better than their stochastic competitors. At the very beginning, Prox-SVRG has a more rapid descent while ASFW and PSFW could obtain smaller function values later on. We can also observe big swings in SVRF periodically. This is because at the beginning of each epoch, SVRF proceeds with noisy gradients and very large step sizes. According to Theorem 1 in \cite{LH16}, the step size of the first step in every epoch can be as large as $1$.
\begin{center}
\includegraphics[scale=0.5]{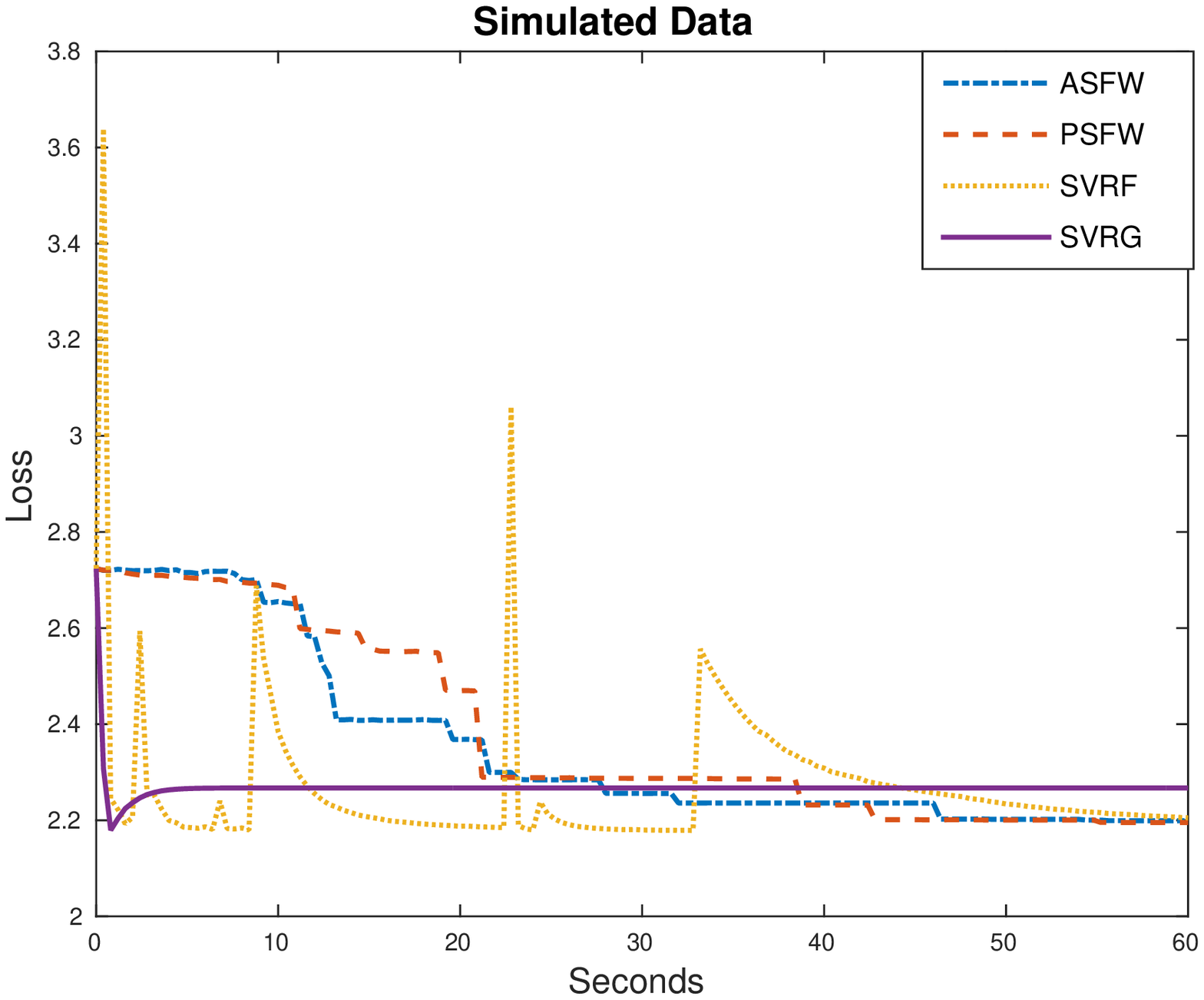}
\end{center}

\subsection{Million Song Dataset}
We implemented ASFW and PSFW for solving least squares problems with elastic-net regularization and tested them on the Million Song Dataset (YearPredictionMSD) \citep{Lich13}\citep{BM11}, which is a dataset of songs with the goal of predicting the release year of a song from its audio features. There are $n = 463,715$ training samples and $p = 90$ features in this dataset. The dataset is the one with largest number of training samples available in the UCI machine learning data repository. Therefore it is interesting to examine the actual performance of stochastic algorithms on such a massive dataset. The least squares with elastic-net regularization model that we used was, 
\begin{align*}
\min_{\mb{x} \in \mr{R}^p}\frac{1}{n} \norm{\mb{Ax} - \mb{b}}^2_2 + \lambda \norm{\mb{x}}_1 + \mu\norm{\mb{x}}^2_2
\end{align*}
where $\mb{A} \in \mr{R}^{n \times p}$ and $\mb{b} \in \mr{R}^n$. $\mu \geq 0$ and $\lambda \geq 0$ are regularization parameters. In the numerical experiments, we considered the constrained version of the problem, that is, 
\begin{align*}
&\text{minimize}\quad \; \frac{1}{n} \norm{\mb{Ax} - \mb{b}}^2_2 + \mu\norm{\mb{x}}^2_2 \\
&\text{subject to}\quad \norm{\mathbf{x}}_1 \leq \alpha
\end{align*}
where $\alpha > 0$ is inversely related to $\lambda$.

\noindent We also compared the ASFW and PSFW with SVRF and Prox-SVRG. We followed the same settings in this real data experiment as that in the simulated data experiment except that we used explicit solutions for solving linear optimizations over an $l_1$-balls in FW algorithms and we used the algorithm in \cite{duchi08} for the solving projections onto $l_1$-balls in the Prox-SVRG algorithm instead of using GUROBI for solving linear optimizations and projections. To make fair comparisons, we use the same starting point for all four algorithms. The logarithm of the loss functions using ASFW, PSFW and Prox-SVRG and the running minimum using SVRF are plotted against CPU time. The figures indicate that the performance of ASFW and PFW is as well as or better than Prox-SVRG and SVRF under different regularization parameters. We also observed huge swings in SVRF periodically in these experiments. Therefore we plot the running minimums instead of the most recent function values for SVRF.

\begin{center}
\includegraphics[scale=0.2]{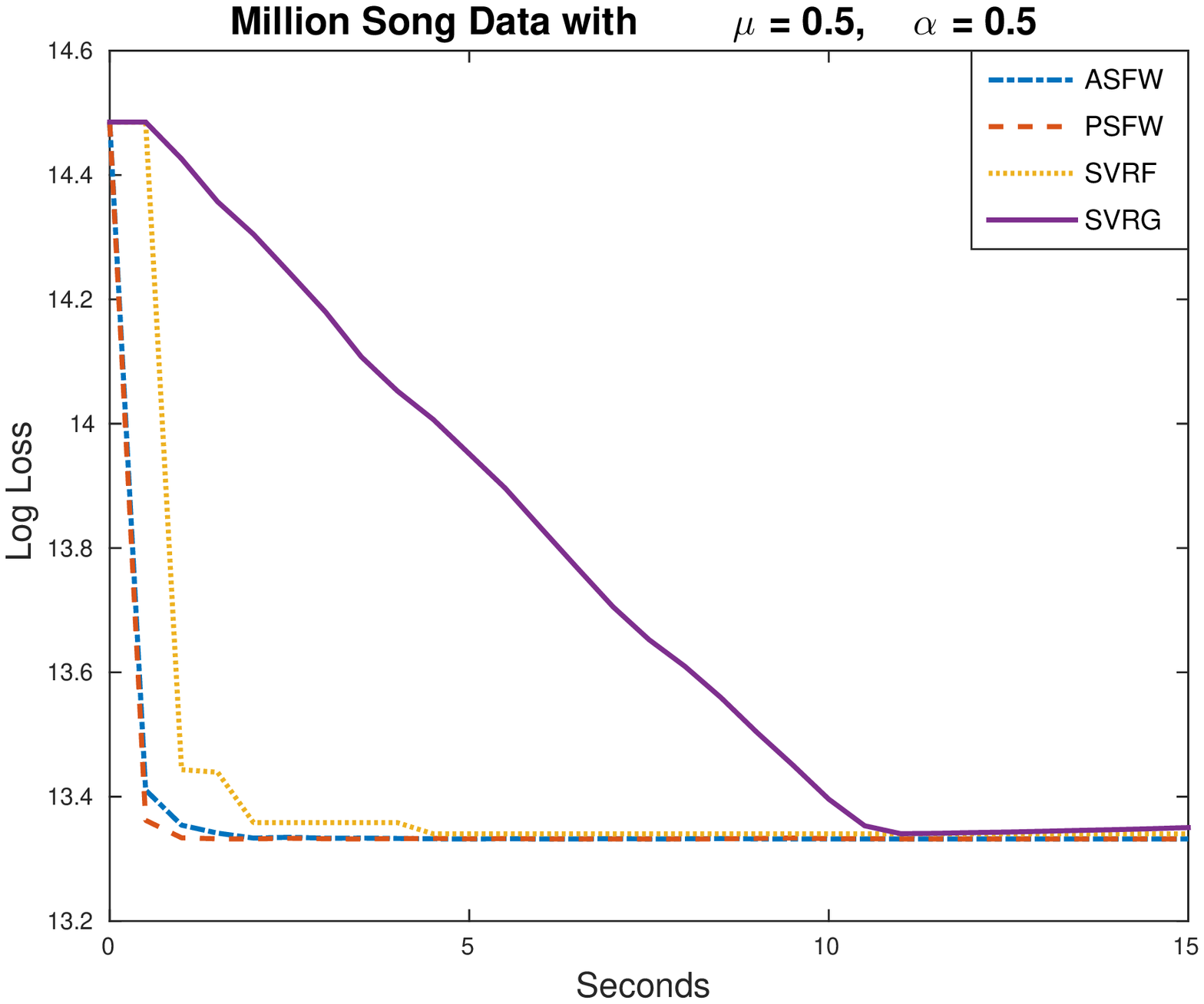}
\includegraphics[scale=0.2]{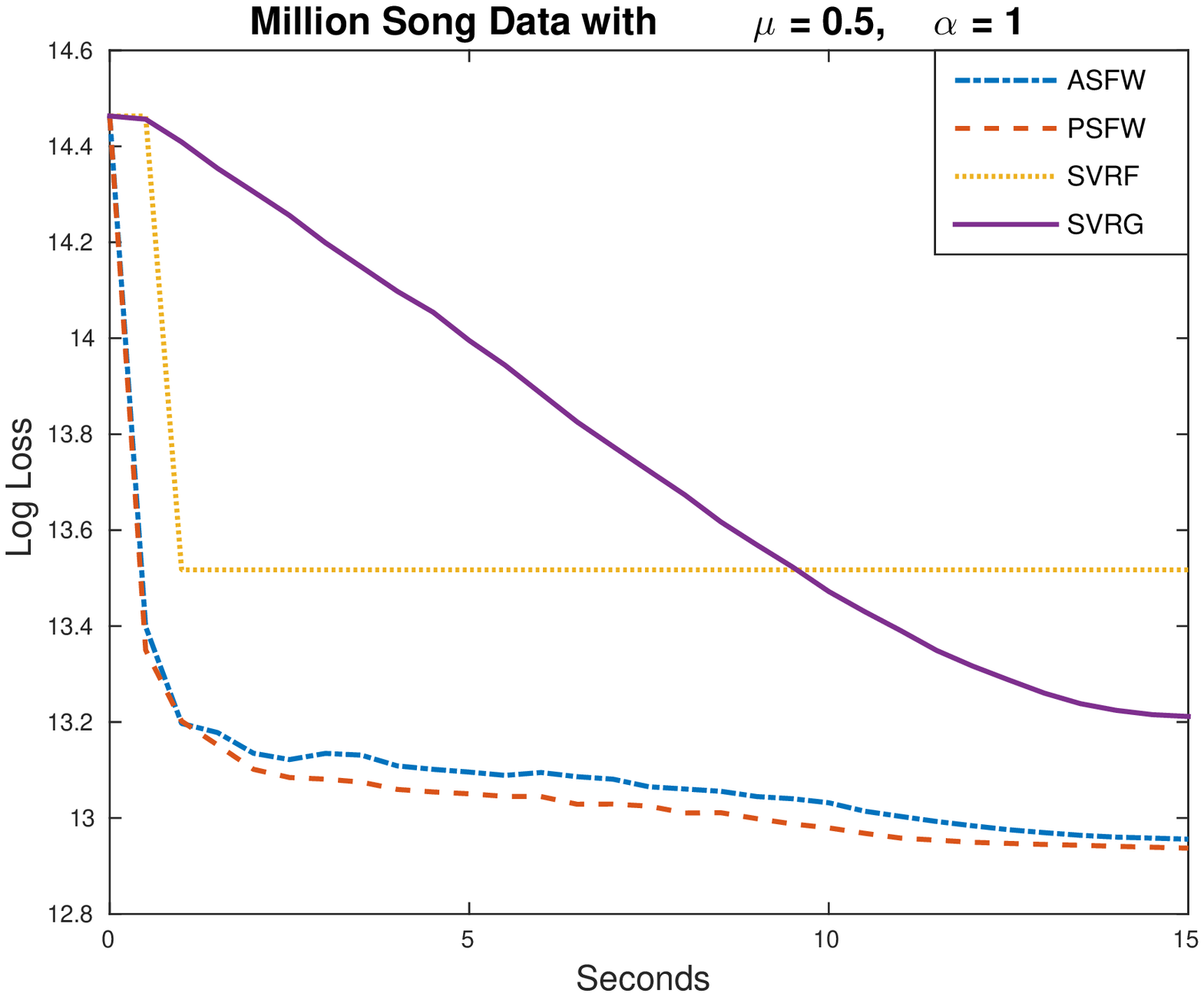}
\includegraphics[scale=0.2]{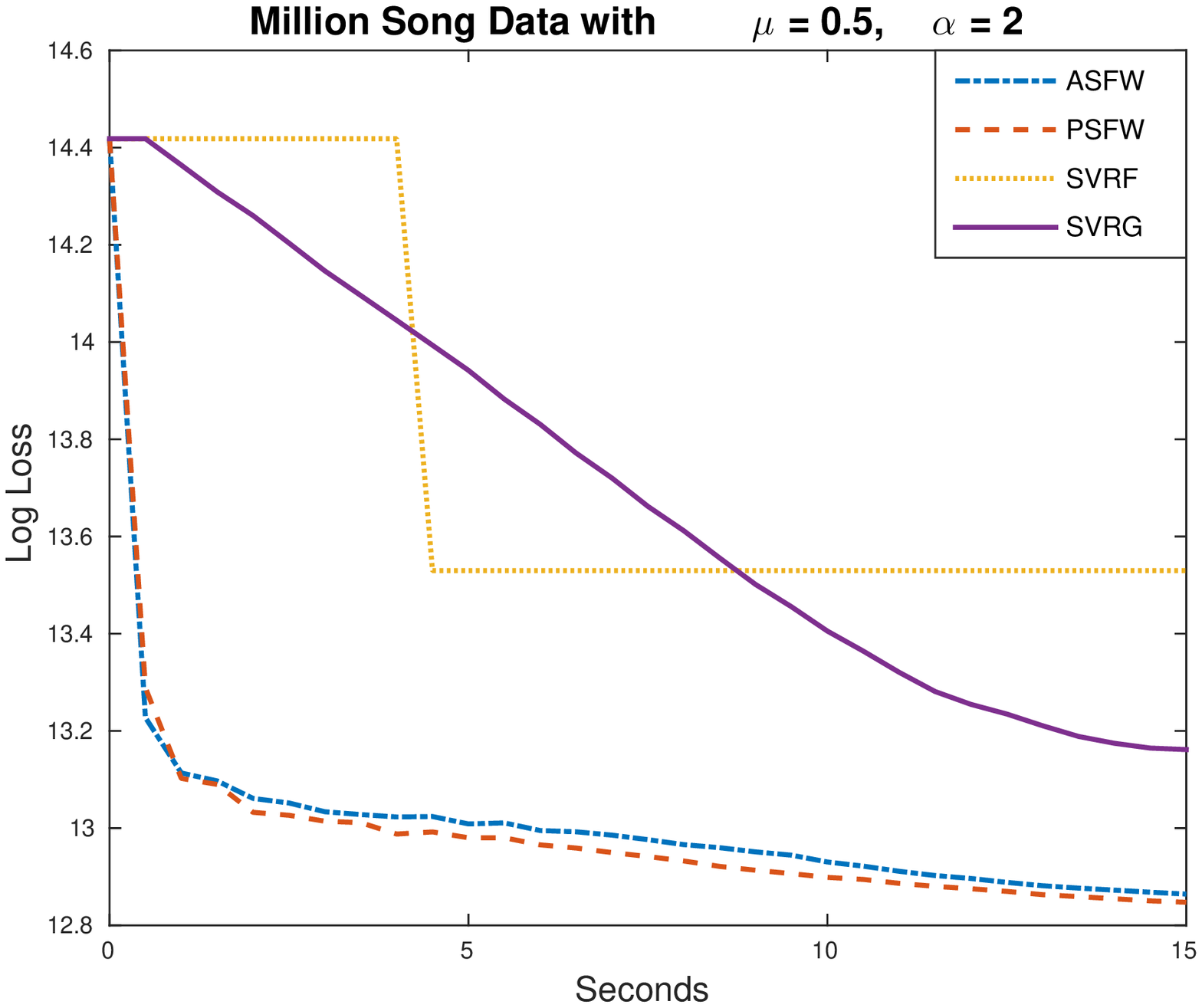}
\includegraphics[scale=0.2]{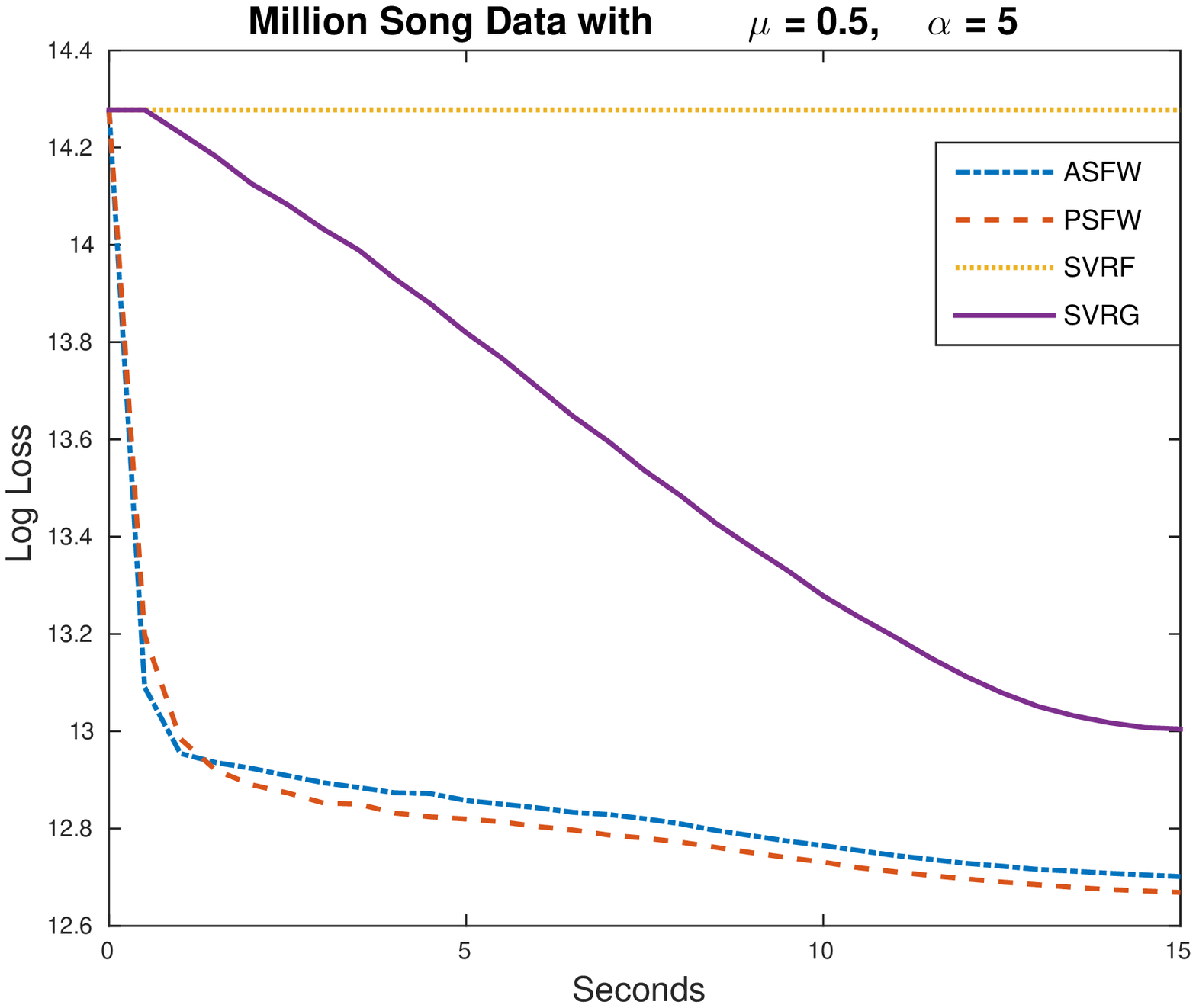}
\includegraphics[scale=0.2]{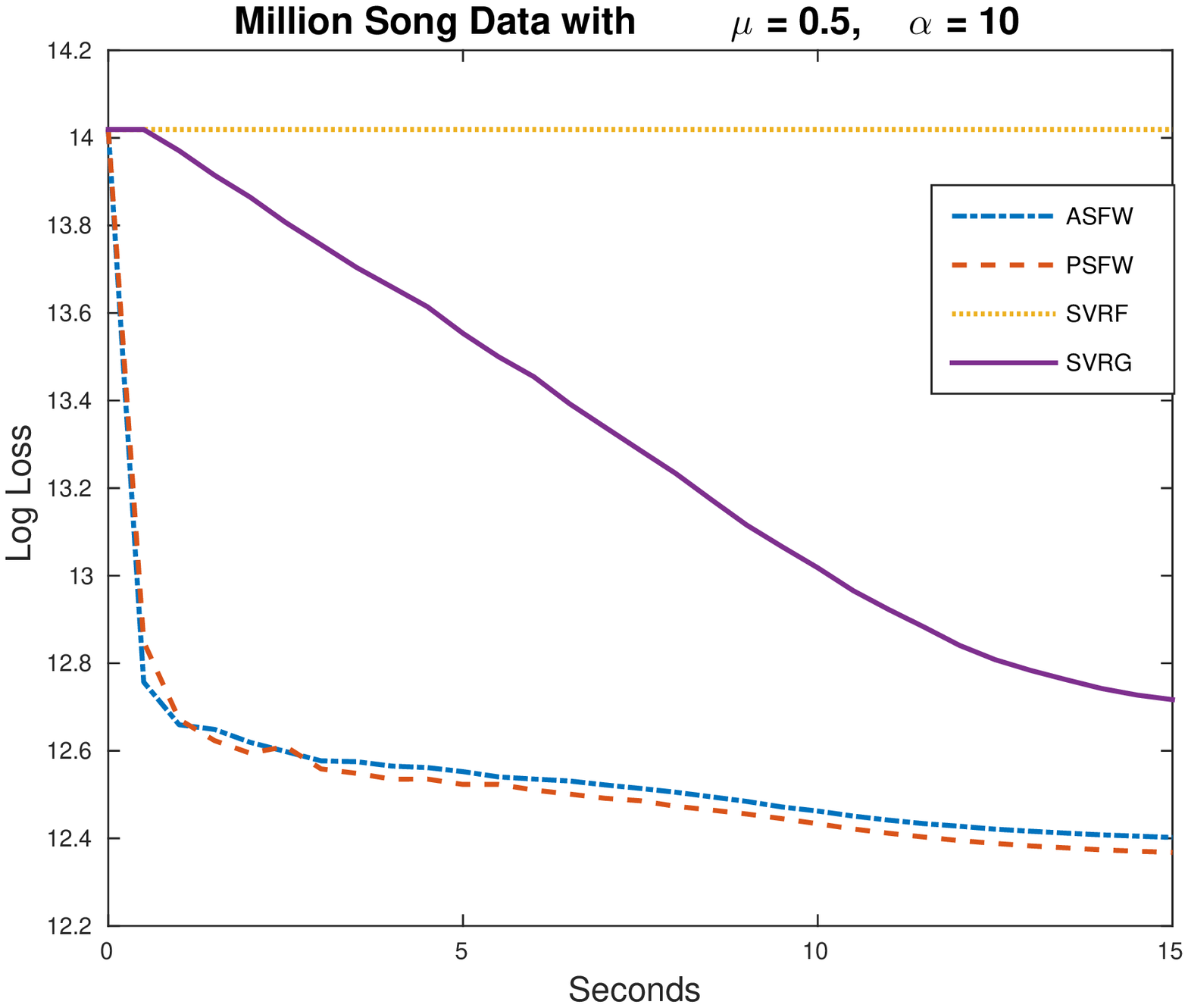}
\includegraphics[scale=0.2]{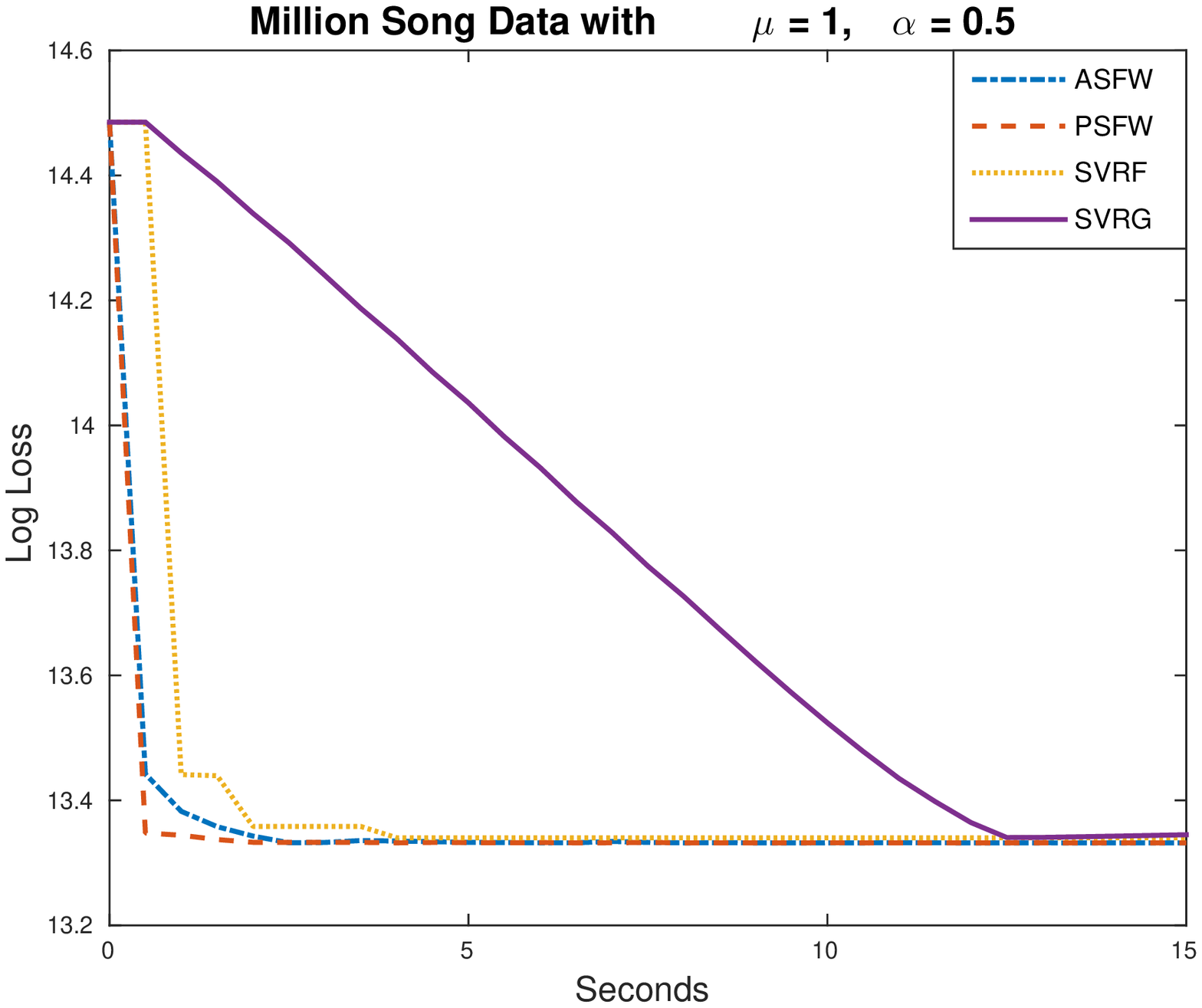}
\includegraphics[scale=0.2]{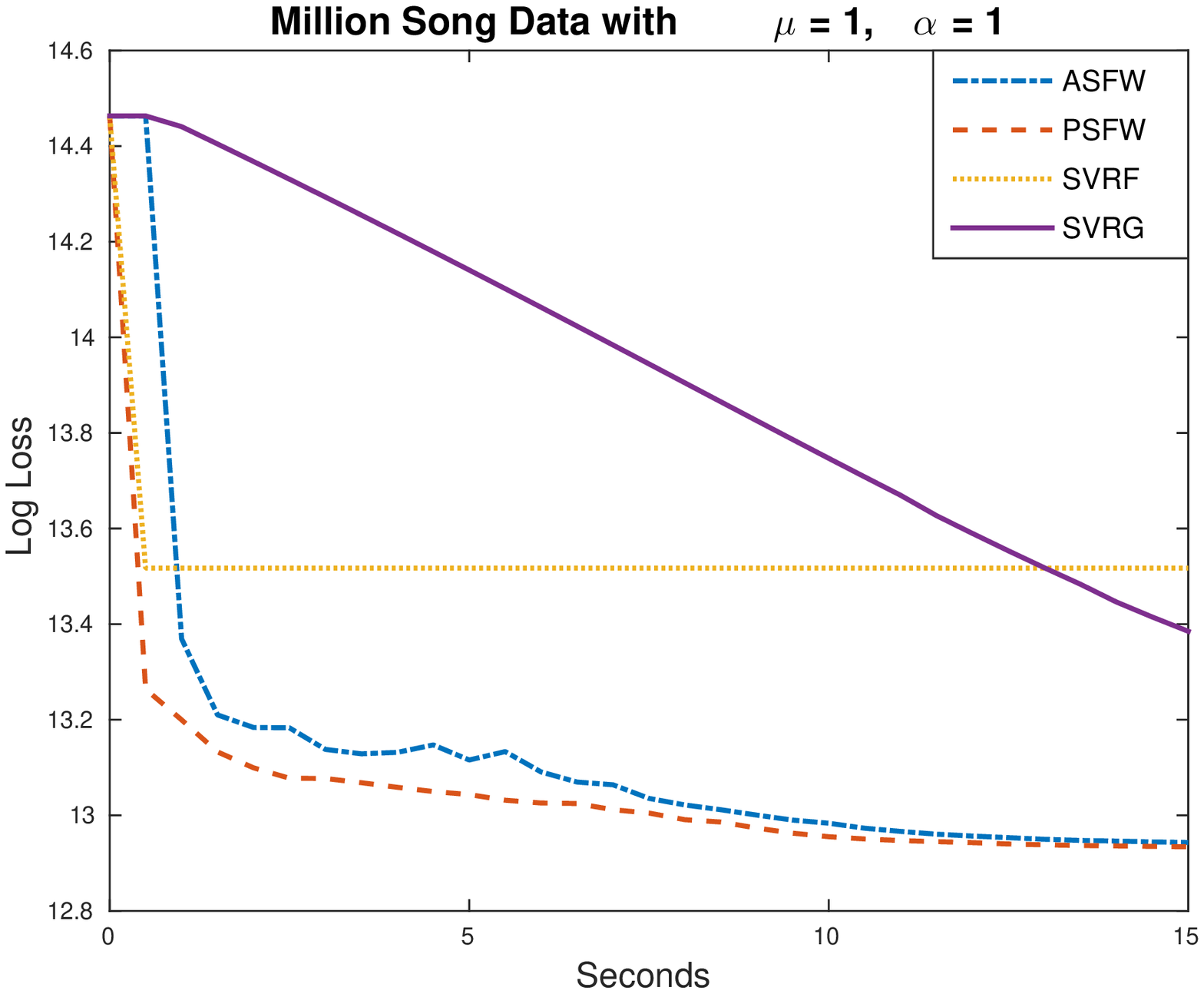}
\includegraphics[scale=0.2]{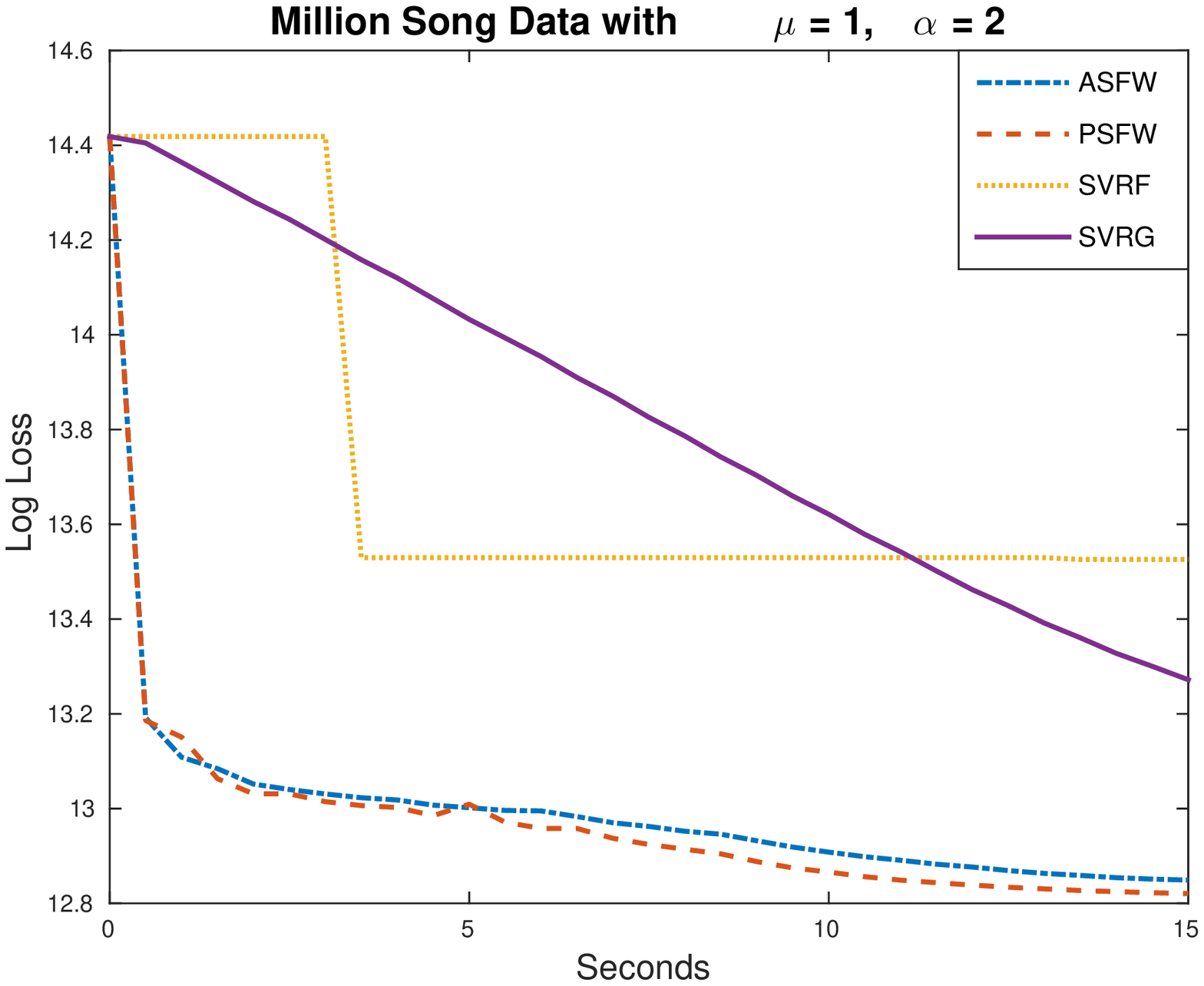}
\includegraphics[scale=0.2]{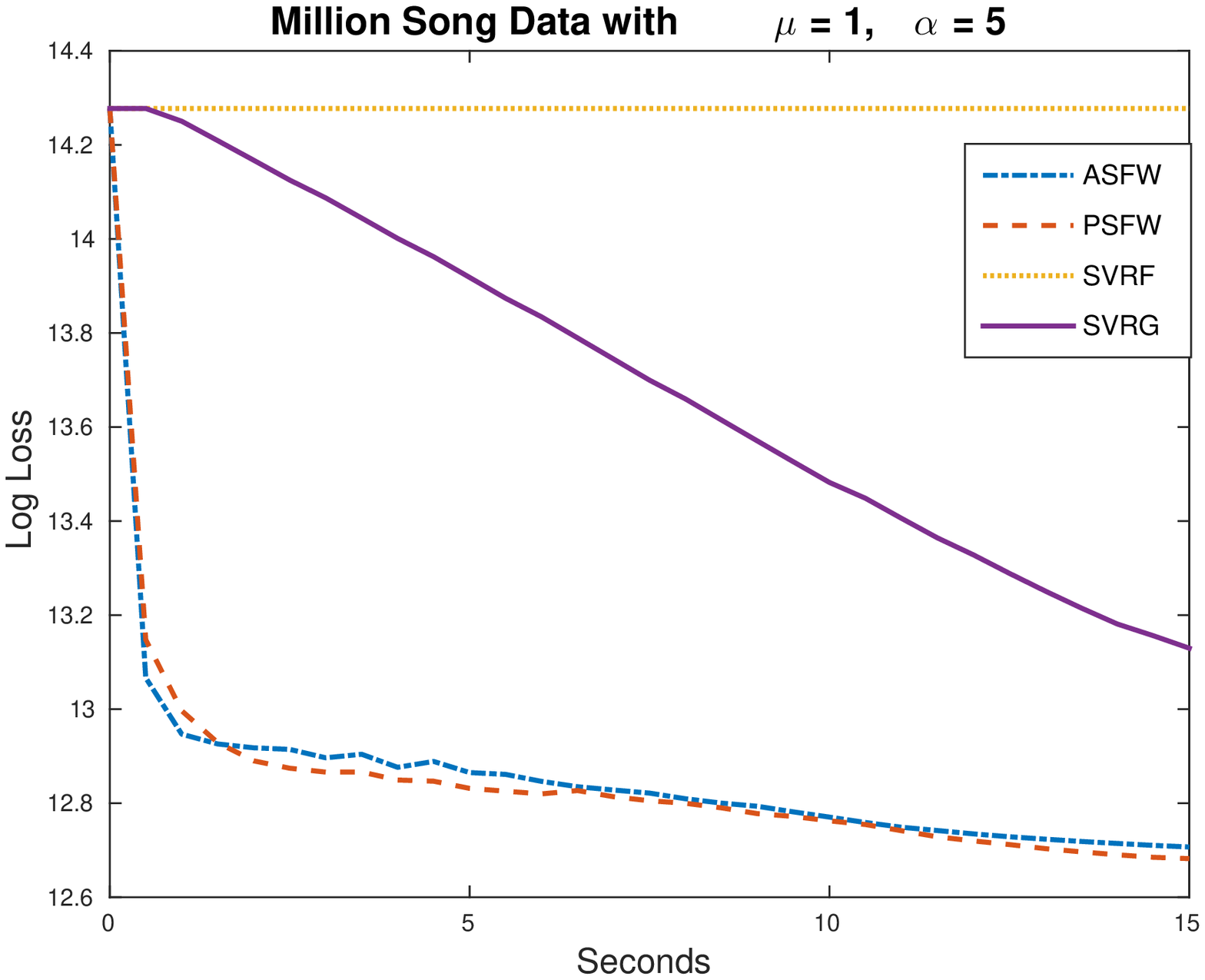}
\includegraphics[scale=0.2]{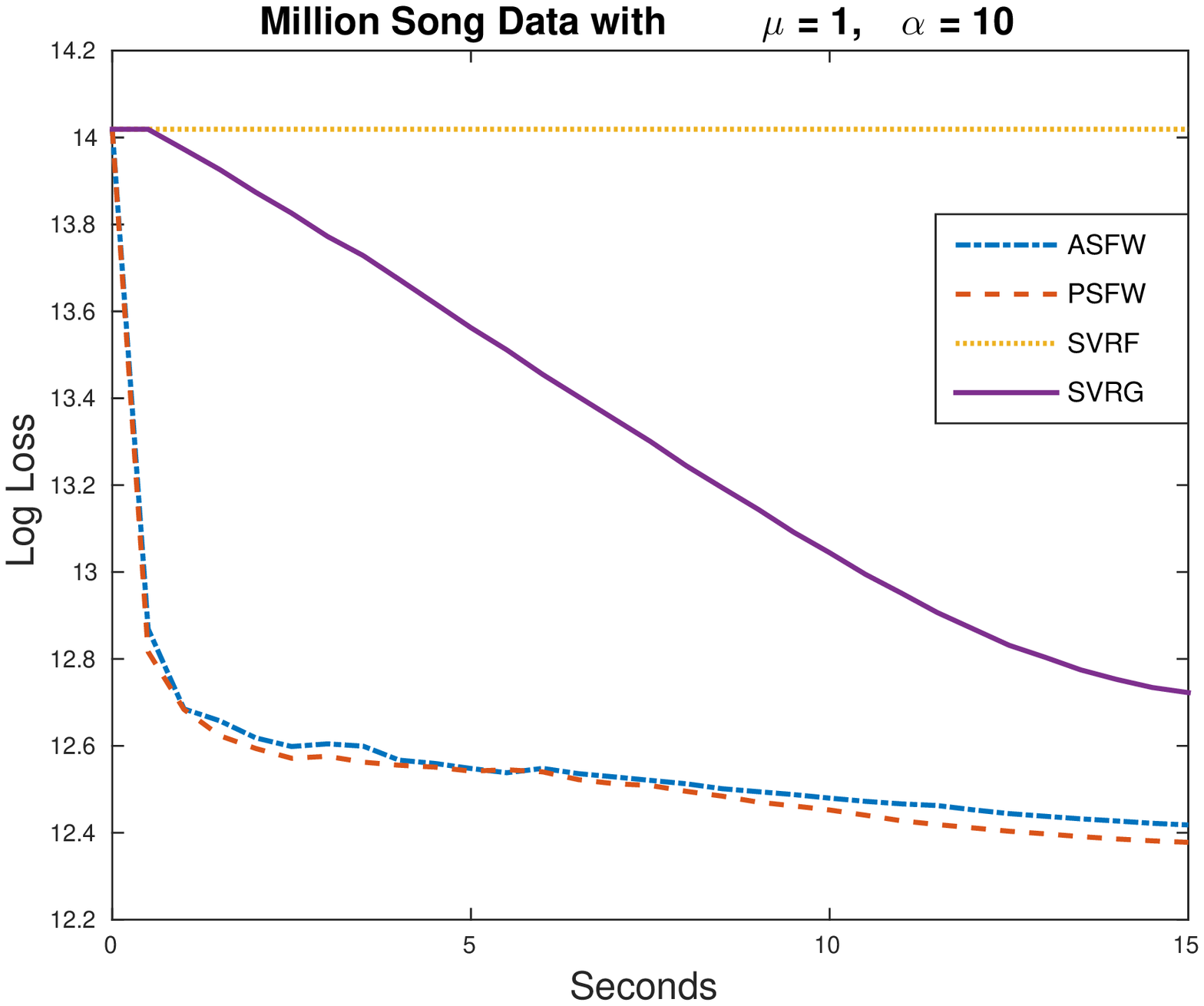}
\includegraphics[scale=0.2]{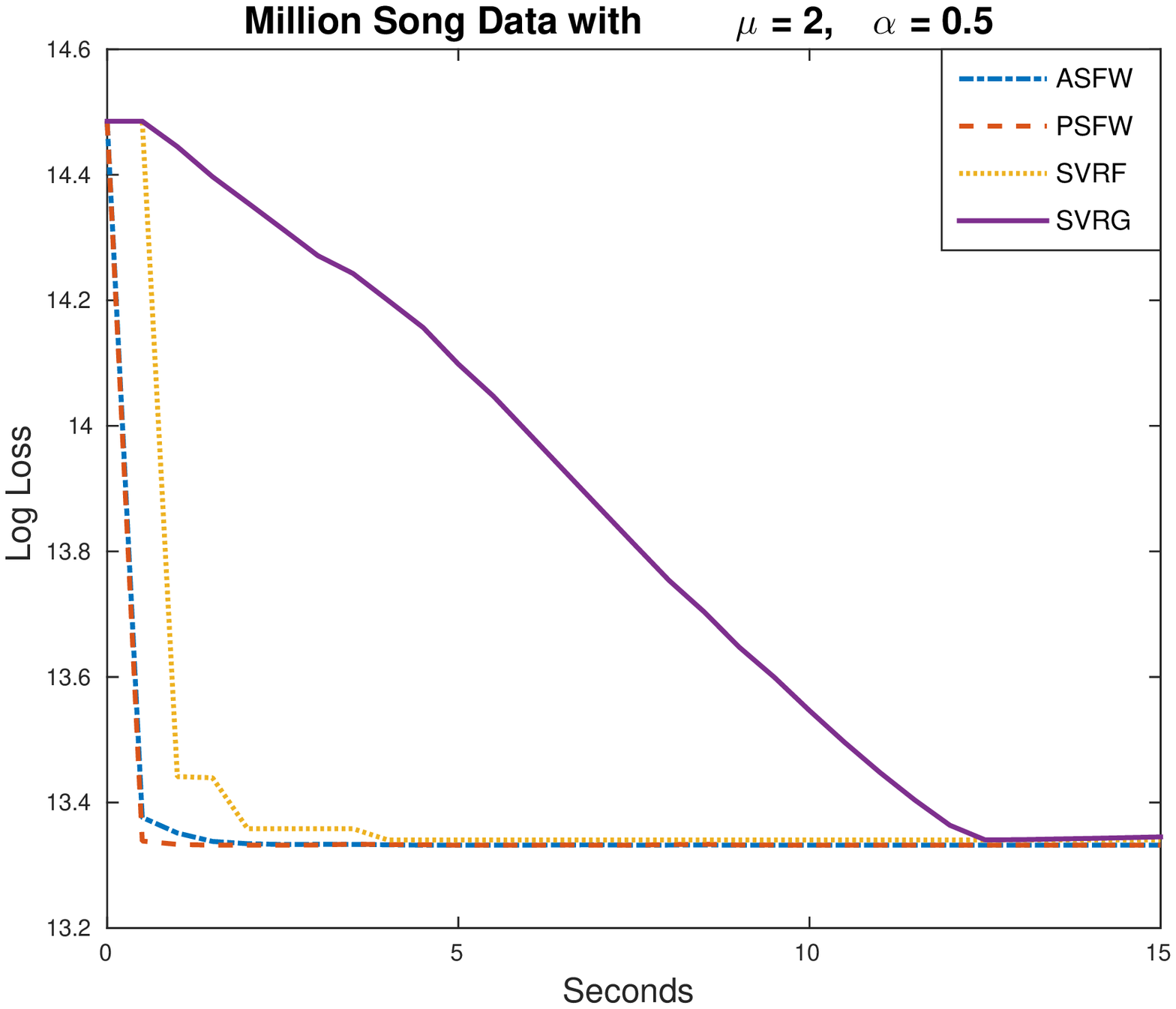}
\includegraphics[scale=0.2]{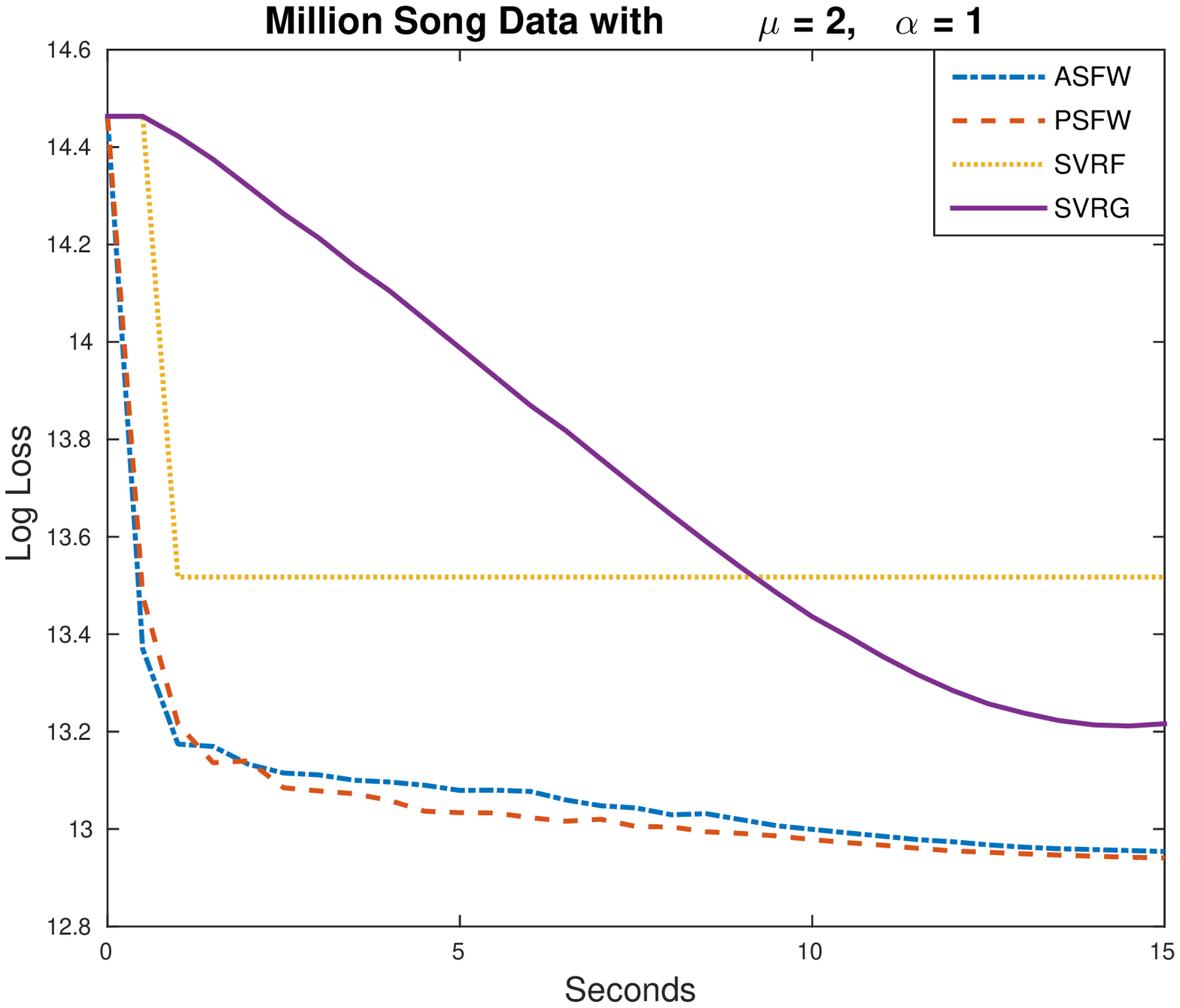}
\includegraphics[scale=0.2]{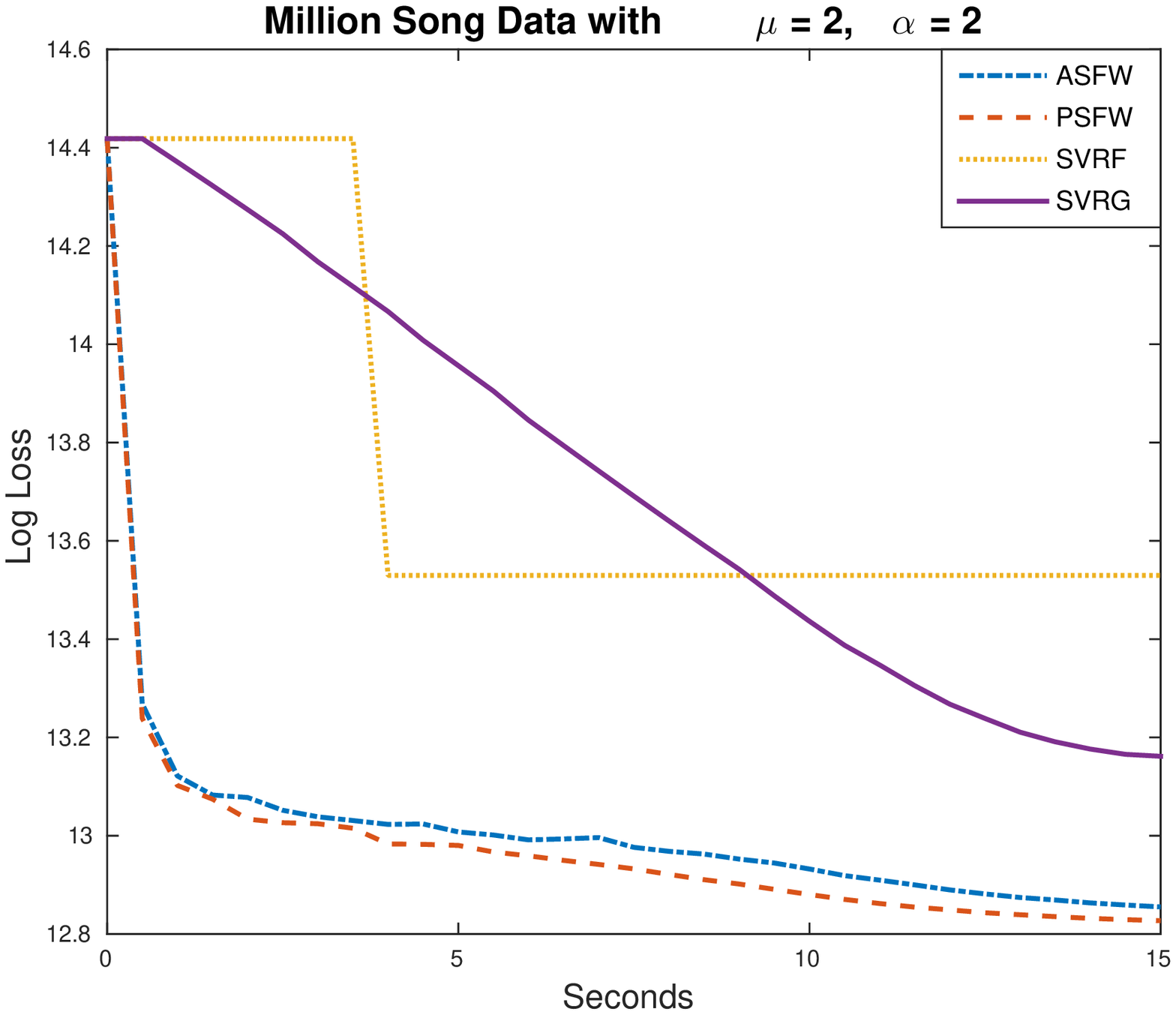}
\includegraphics[scale=0.2]{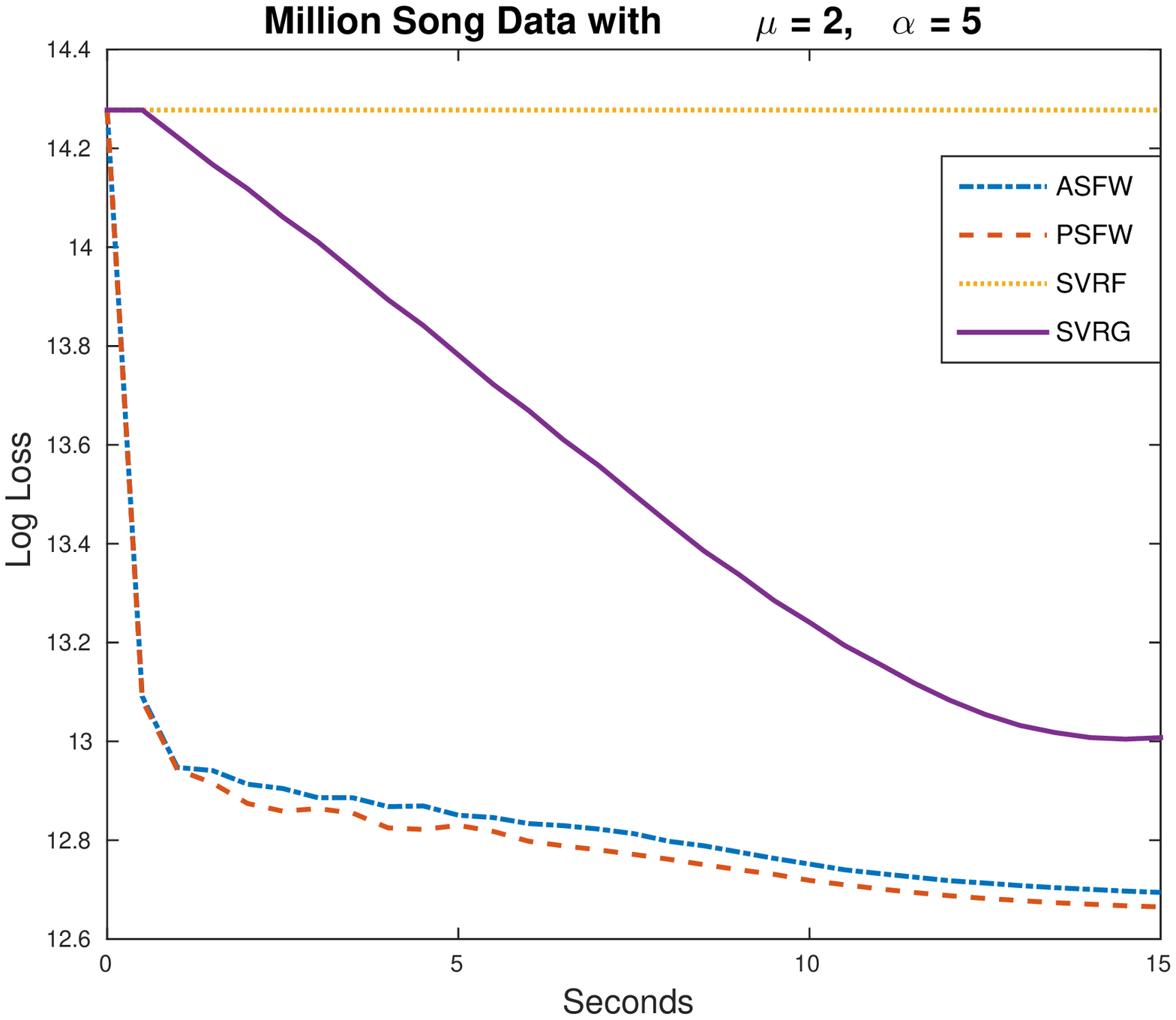}
\includegraphics[scale=0.2]{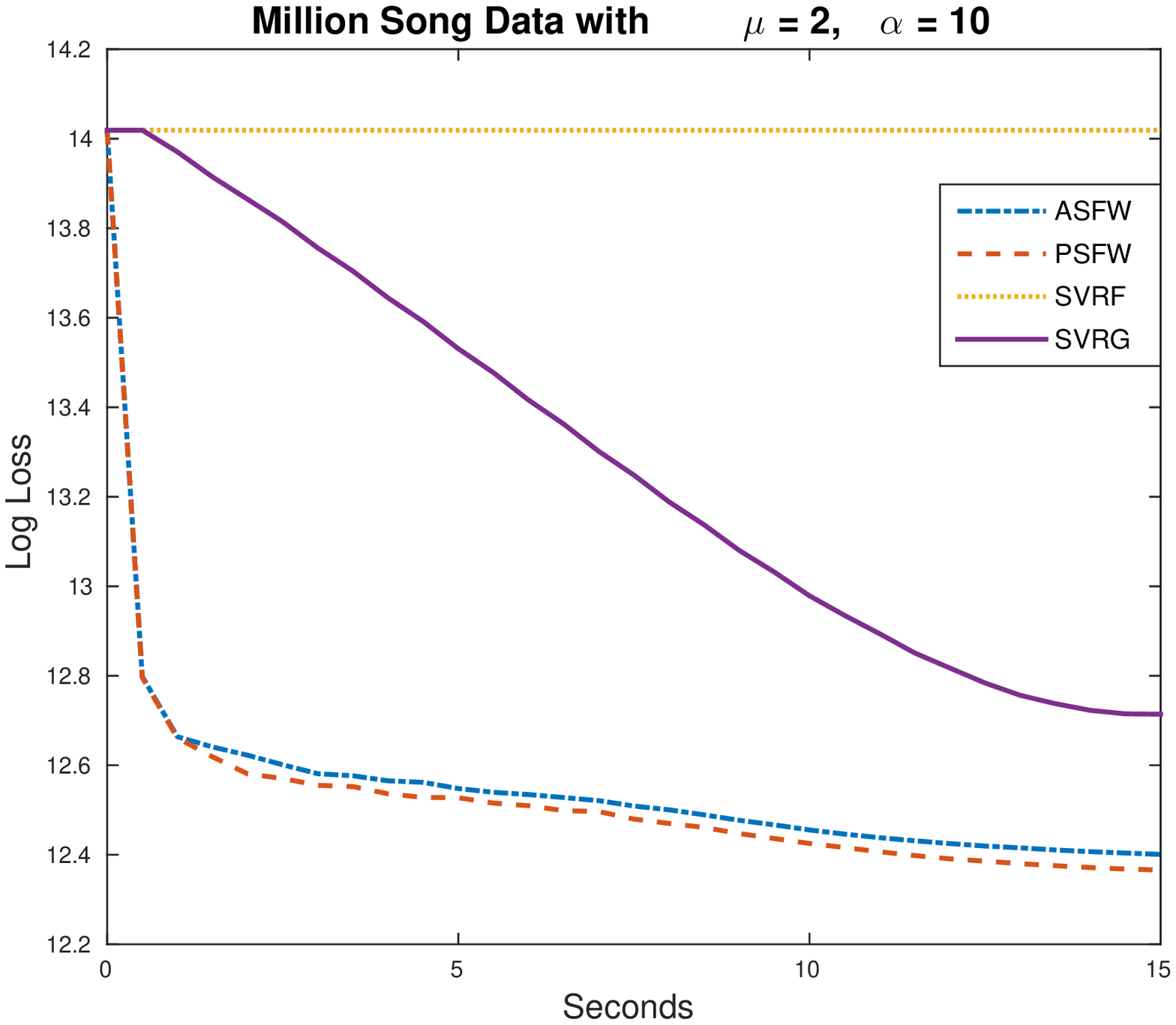}
\includegraphics[scale=0.2]{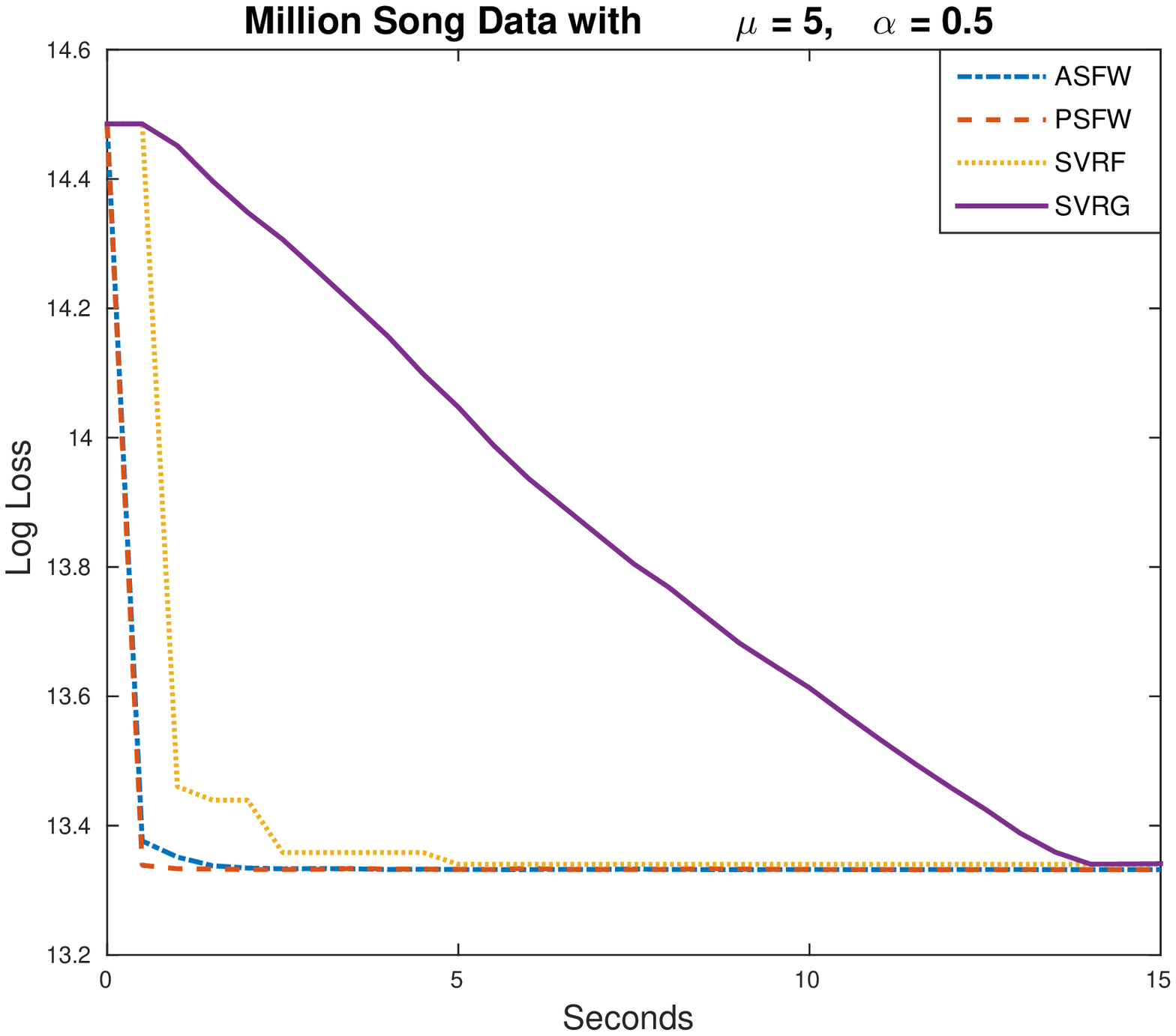}
\includegraphics[scale=0.2]{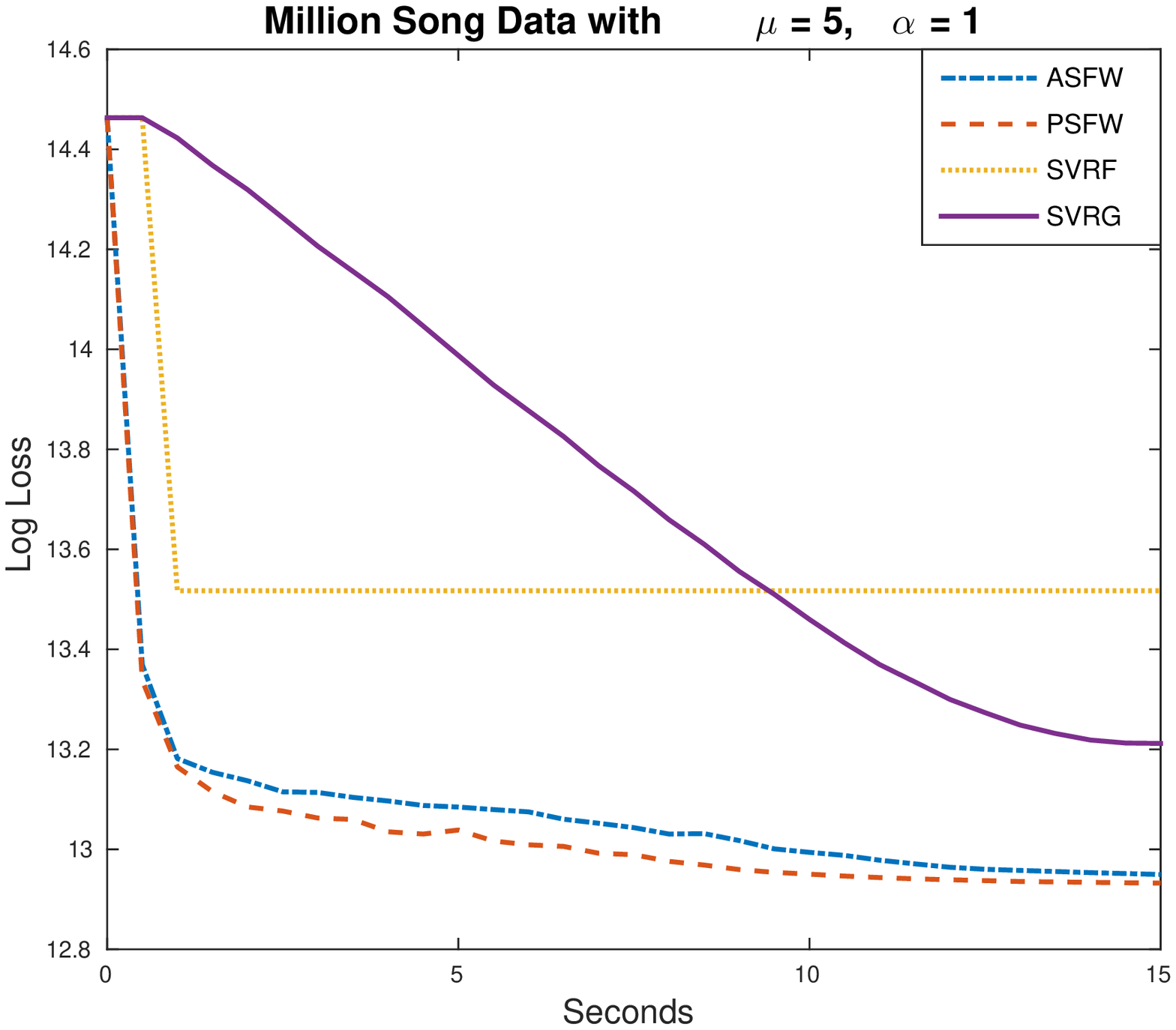}
\includegraphics[scale=0.2]{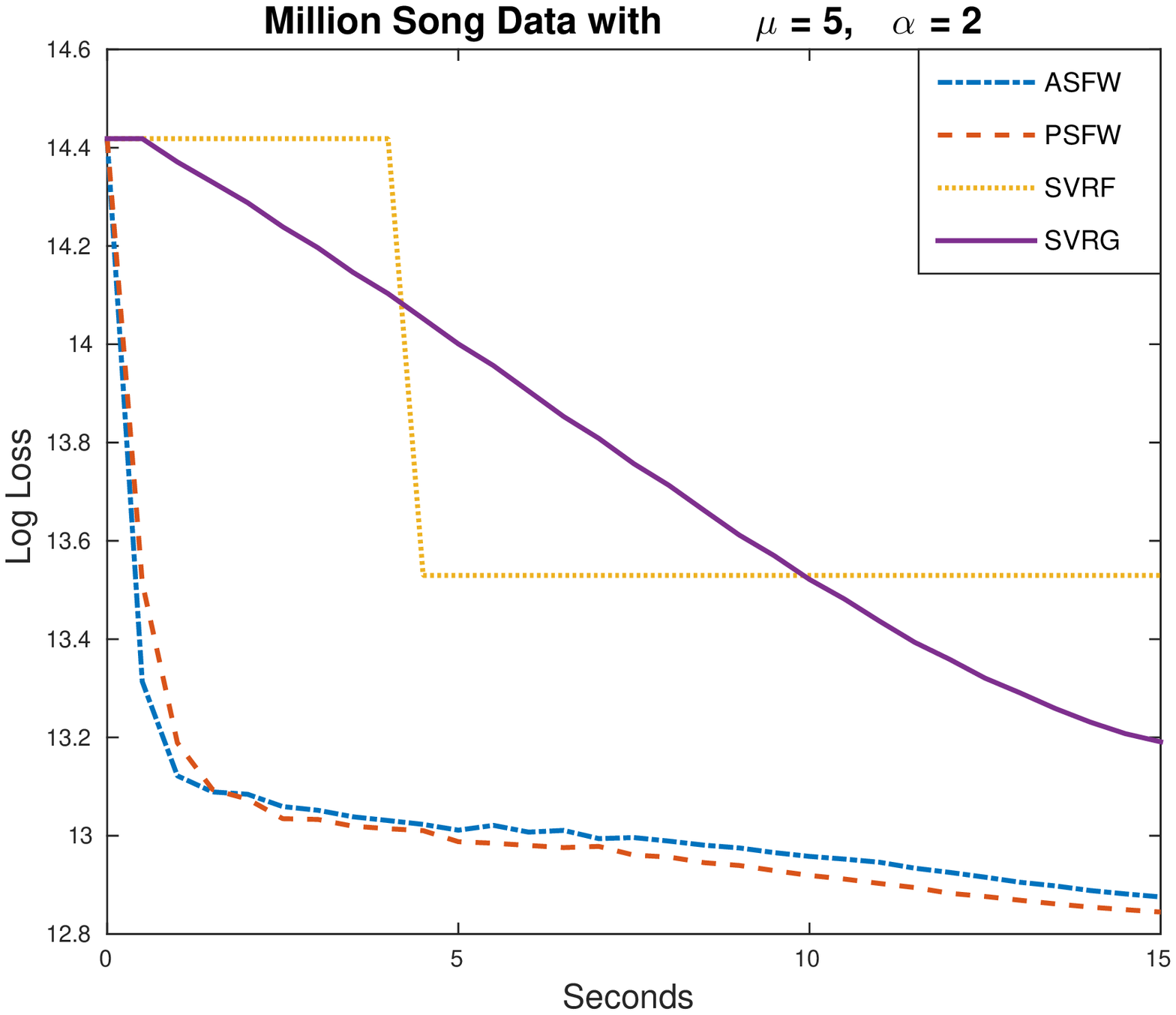}
\includegraphics[scale=0.2]{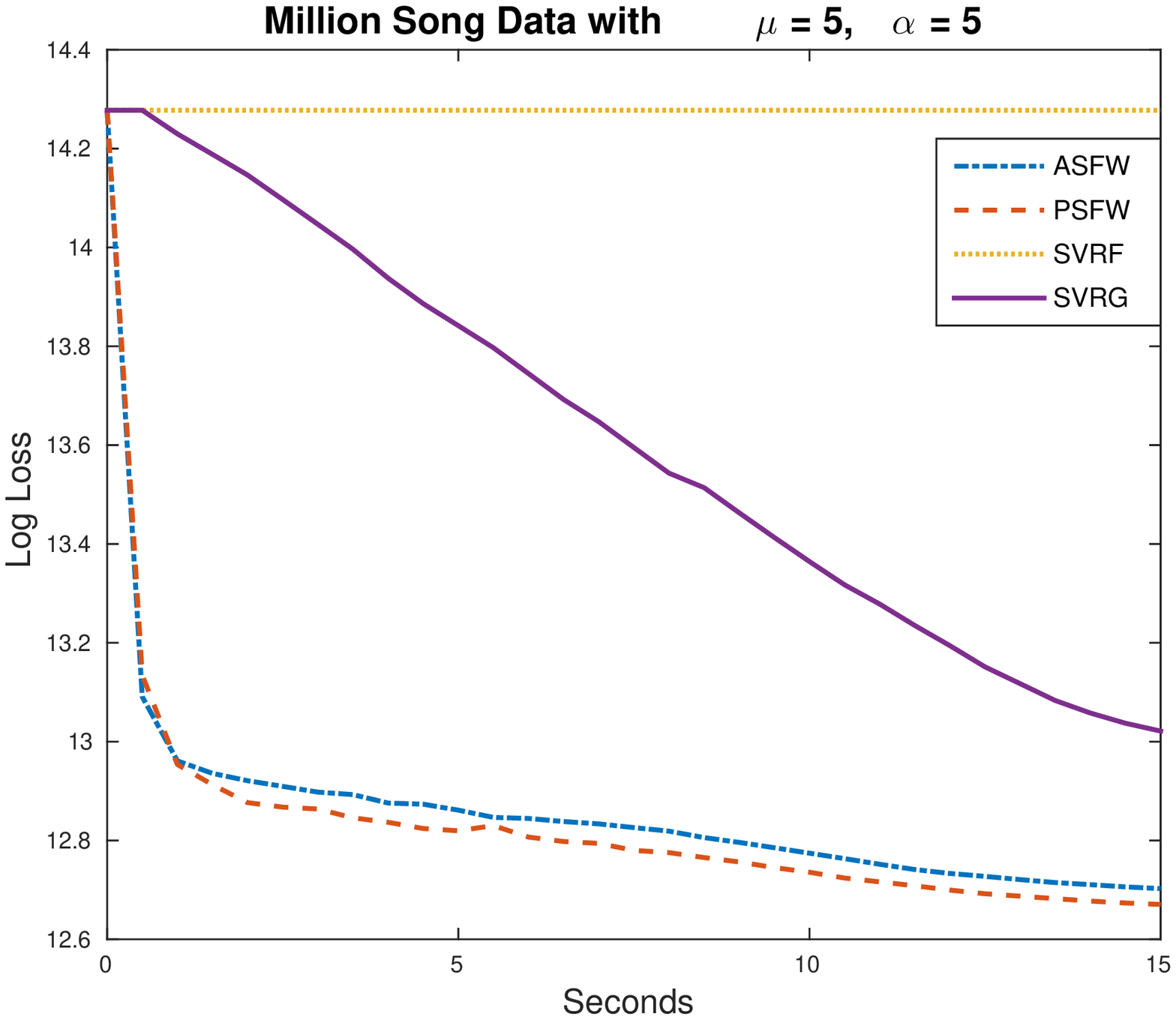}
\includegraphics[scale=0.2]{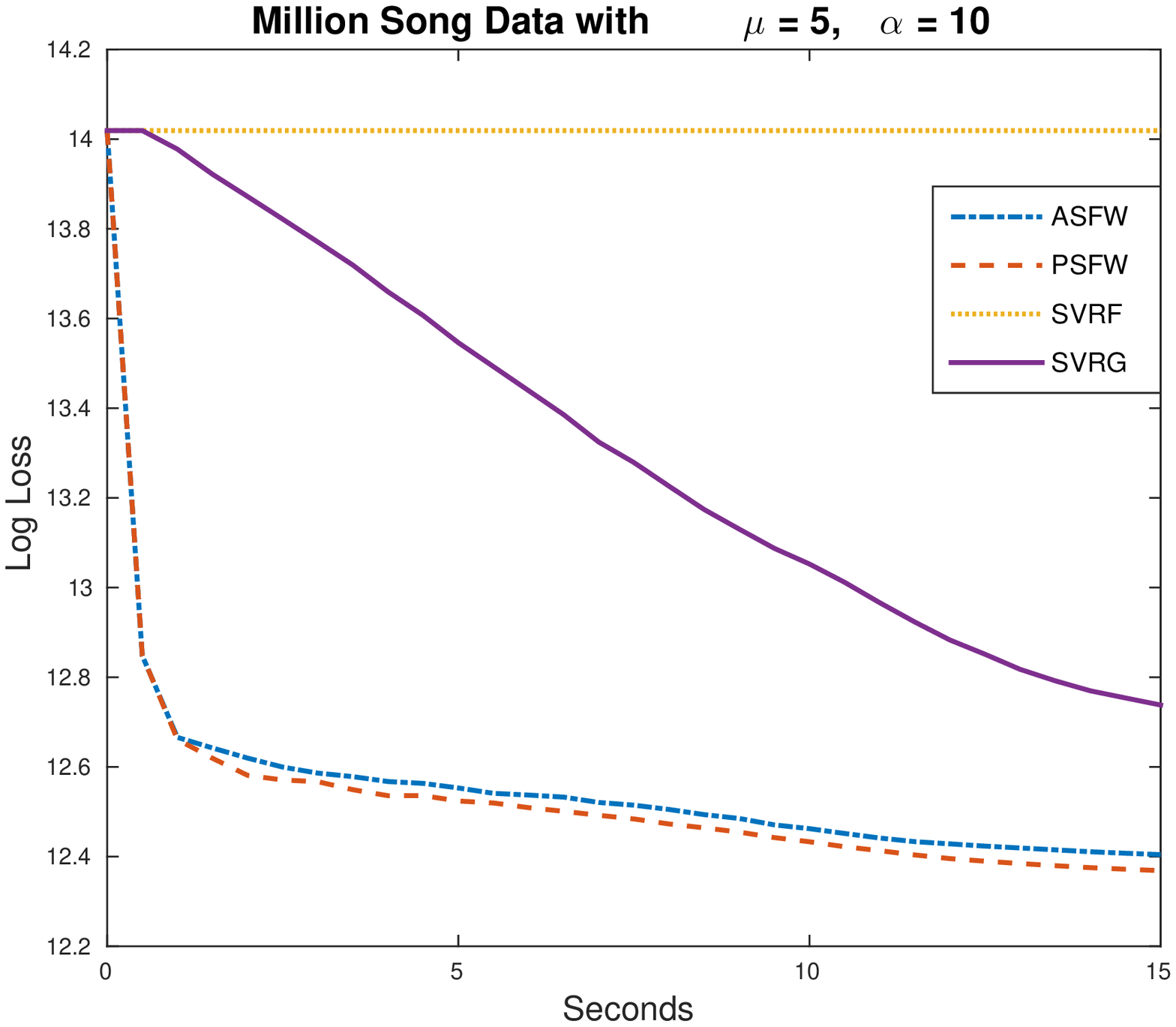}
\includegraphics[scale=0.2]{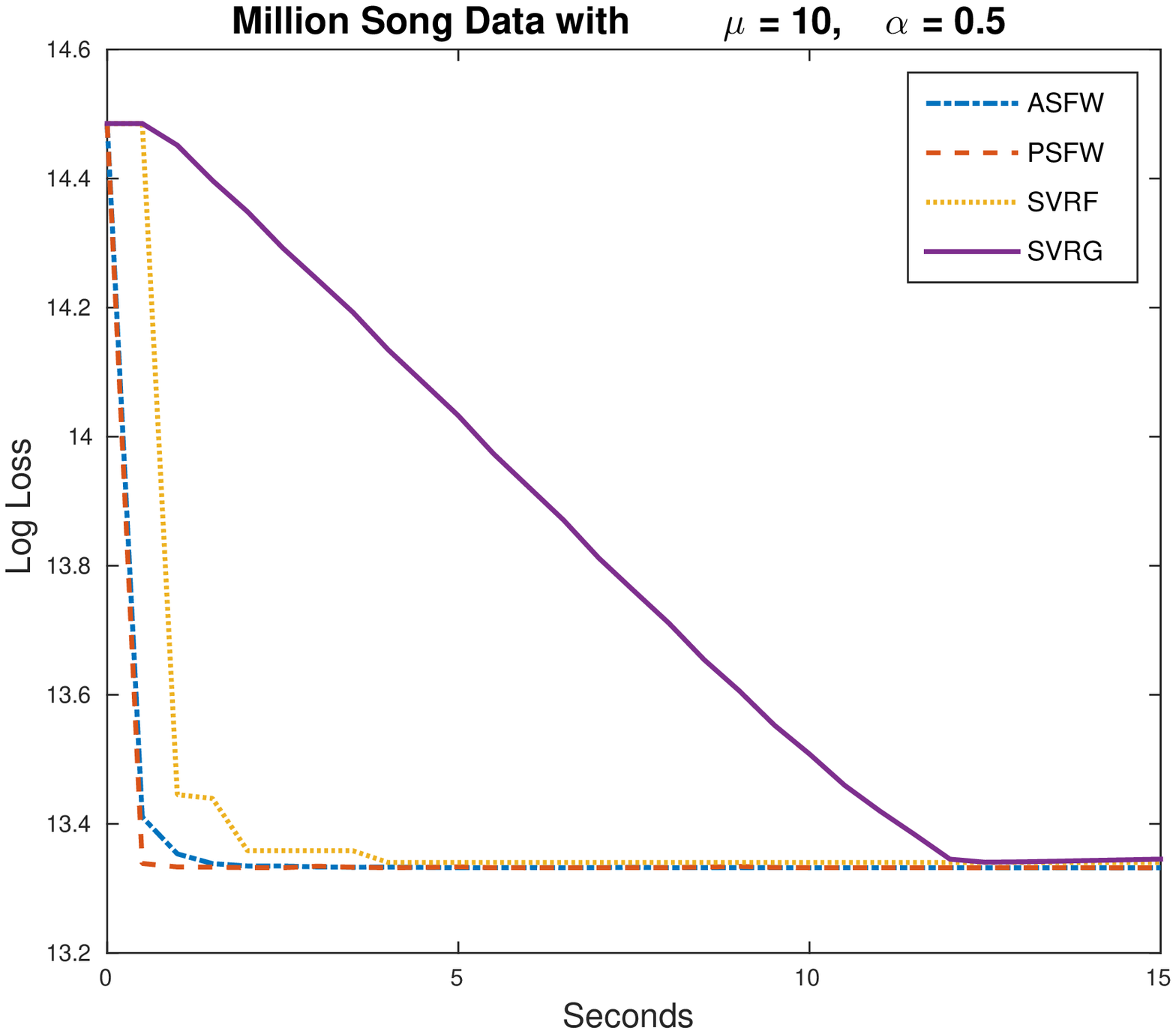}
\includegraphics[scale=0.2]{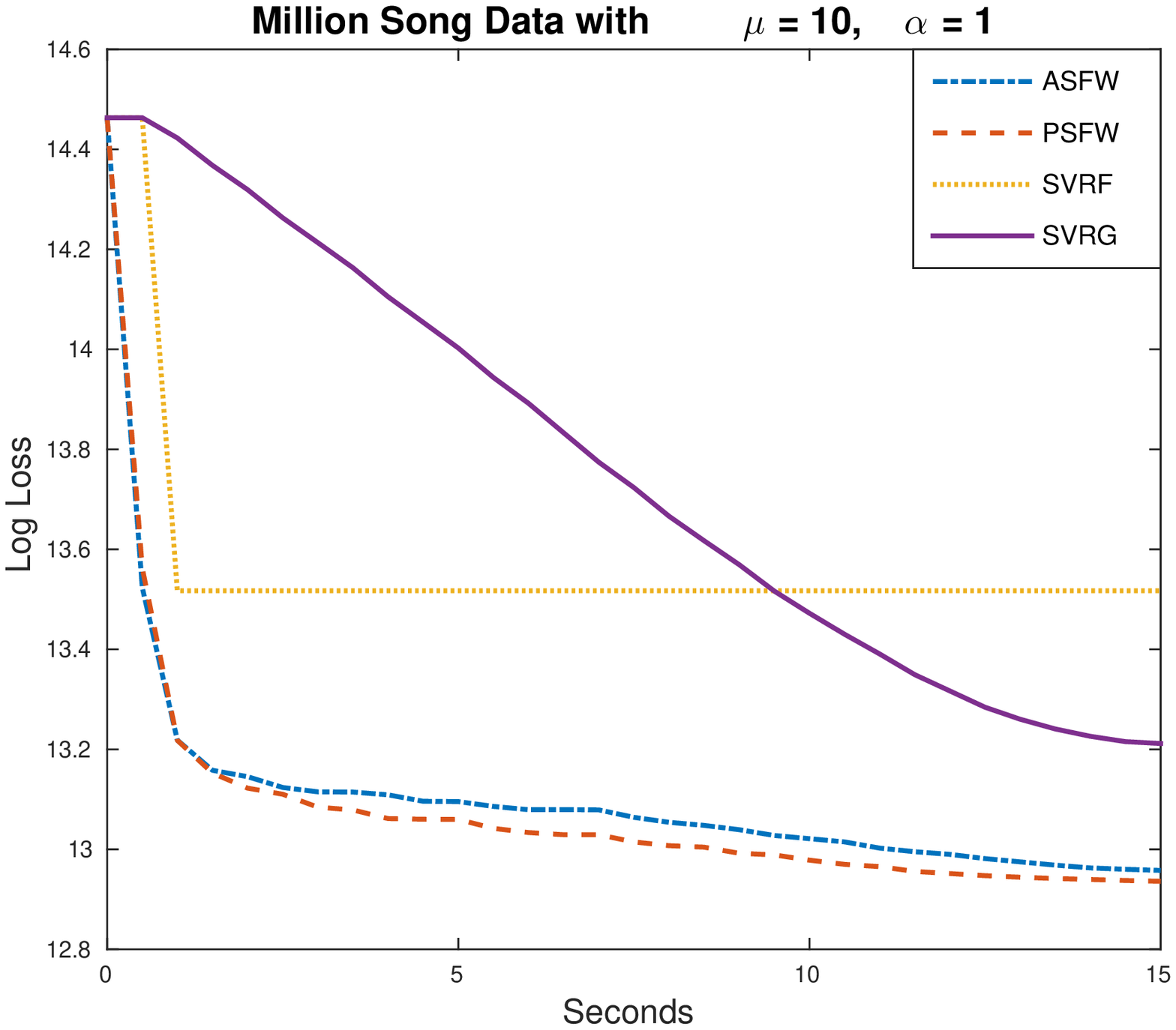}
\includegraphics[scale=0.2]{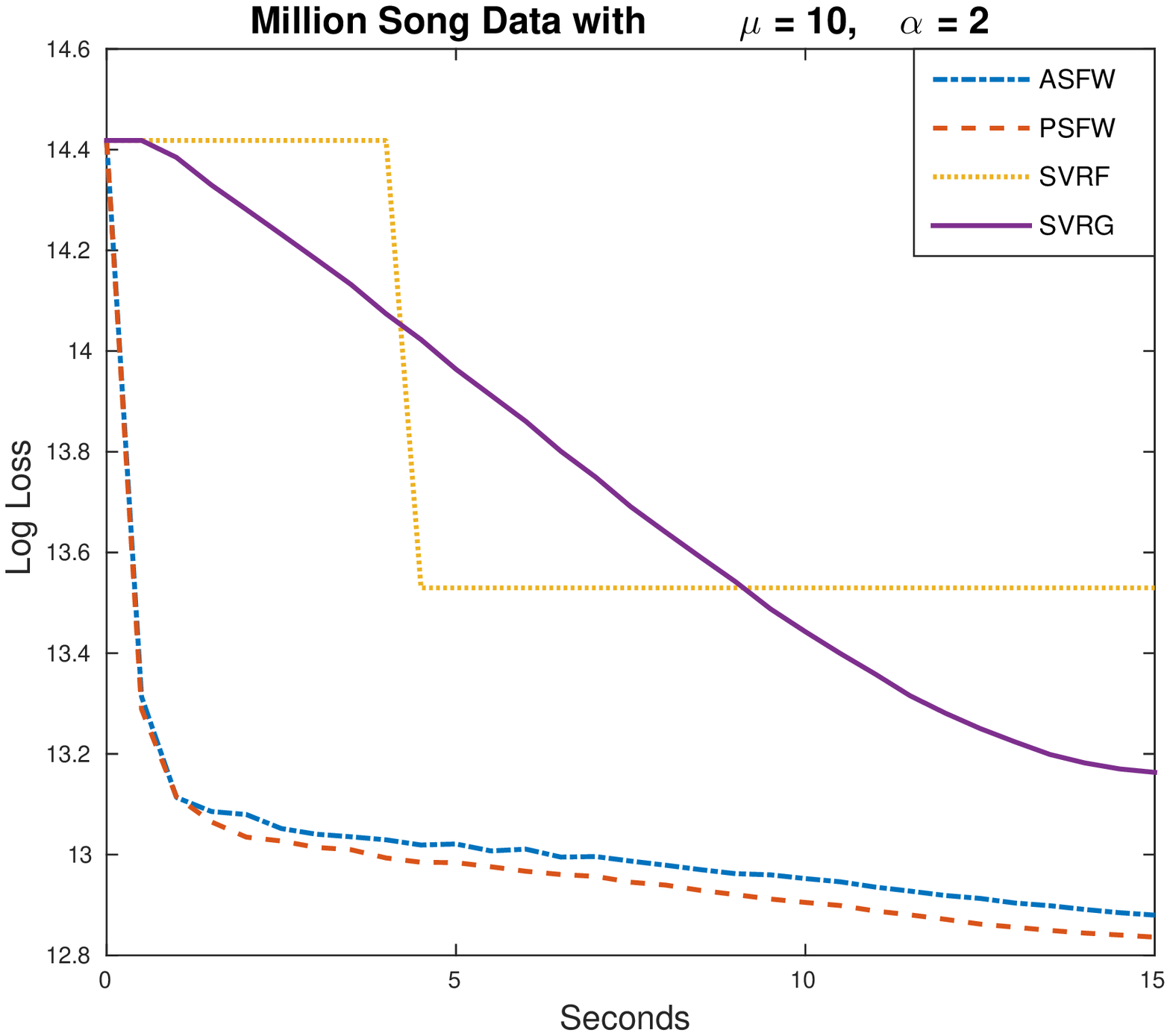}
\includegraphics[scale=0.2]{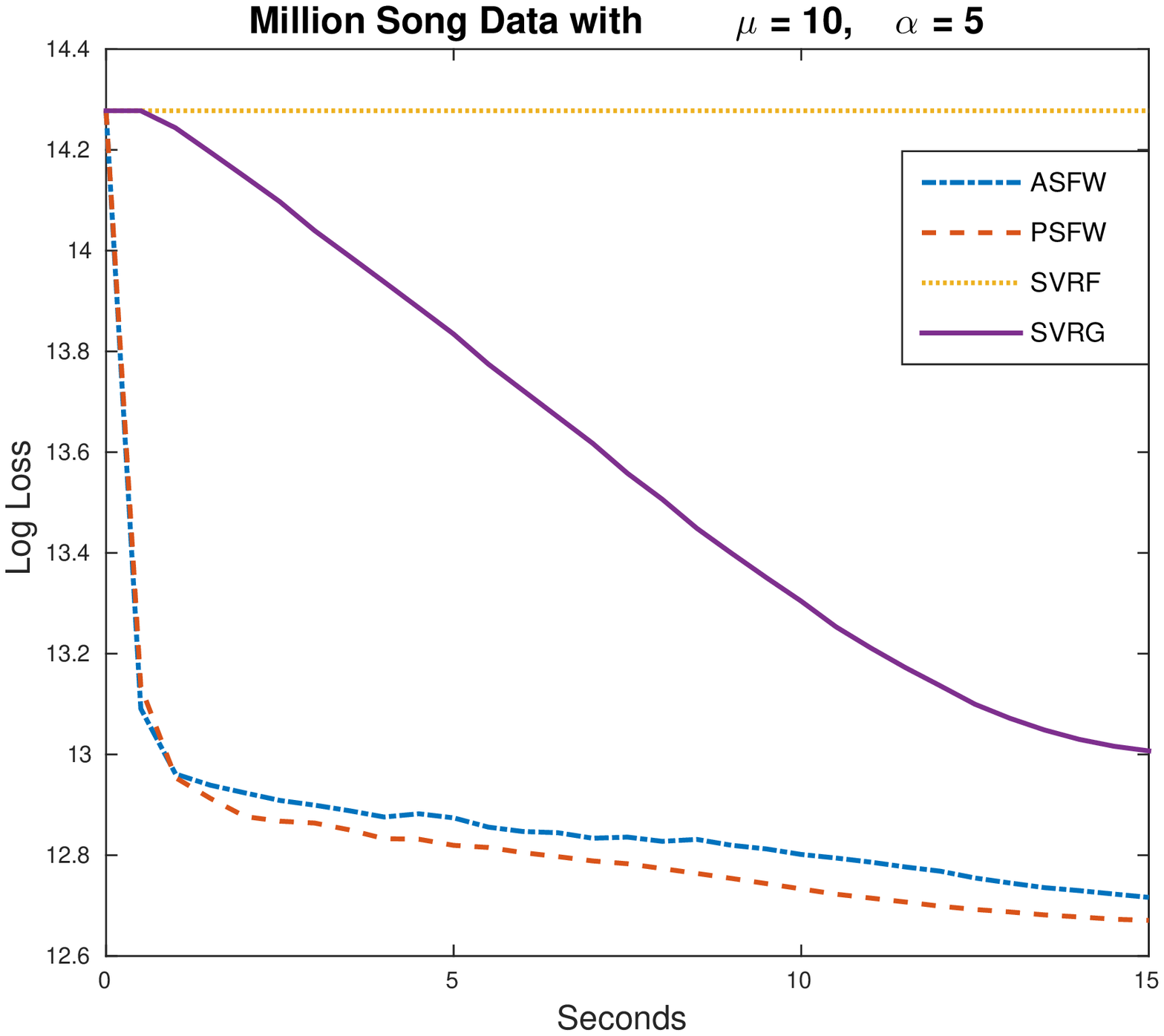}
\includegraphics[scale=0.2]{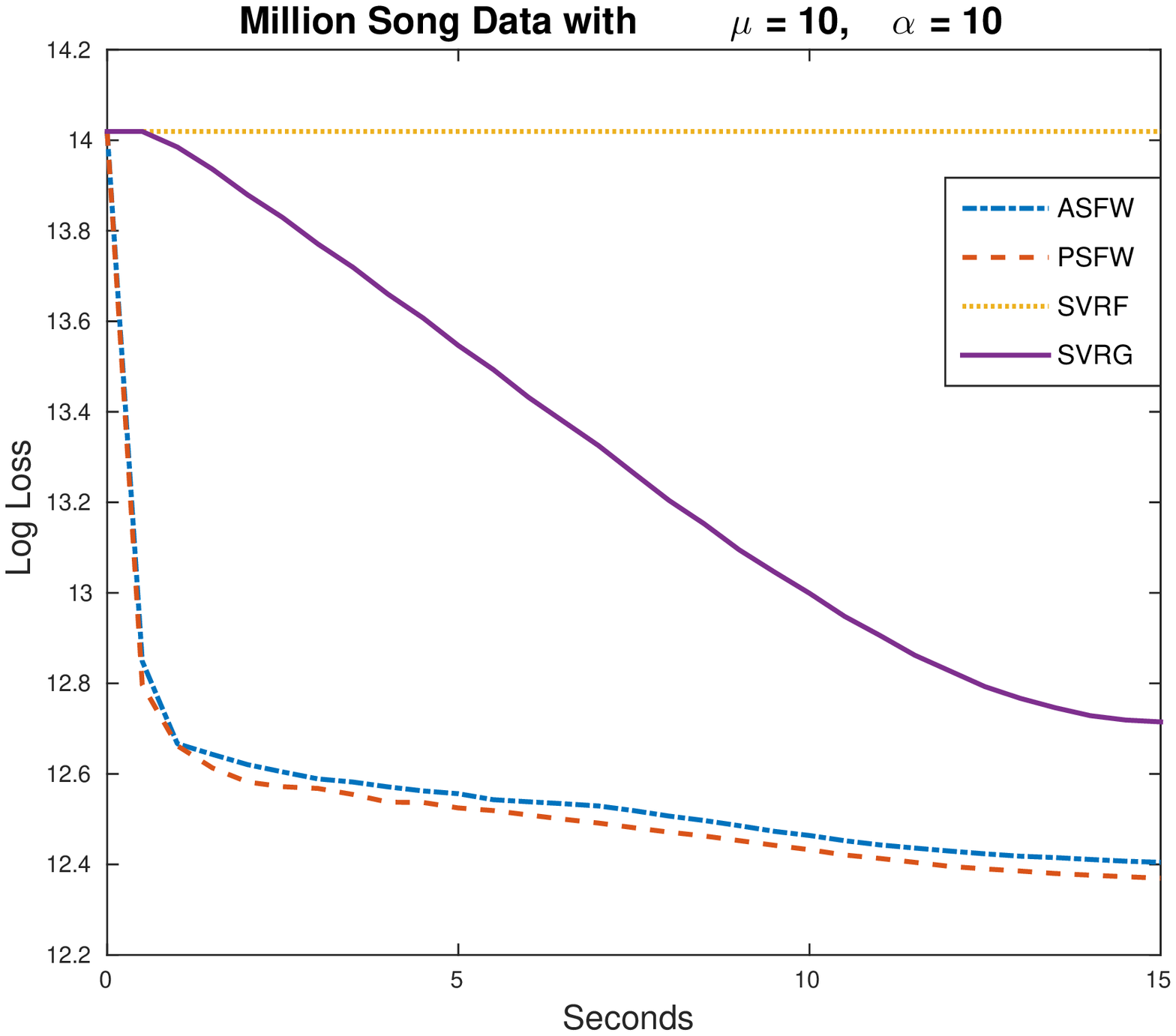}
\end{center}
\section{Conclusion and Future Work}
In this paper, we proved linear convergence almost surely and in expectation of the Away-step Stochastic Frank-Wolfe algorithm and the Pairwise Stochastic Frank-Wolfe algorithm by using a novel proof technique. We tested these algorithms by training a least squares model with elastic-net regularization on the million song dataset and on a synthetic problem. The proposed algorithms performed as well as or better than their stochastic competitors for various choice of the regularization parameters. Future work includes extending the proposed algorithms to problems with block-coordinate structures and non-strongly convex objective functions and using variance reduced stochastic gradients to reduce the number of stochastic gradient oracle calls.
\bibliographystyle{unsrt}
\bibliography{fw_write_up.bib}

\begin{thebibliography}{10}

\bibitem{FW56}
M.~Frank and P.~Wolfe.
\newblock An algorithm for quadratic programming.
\newblock {\em Naval Research Logistics Quarterly}, 3:95–110, 1956.

\bibitem{N15}
Yu. Nesterov.
\newblock Complexity bounds for primal-dual methods minimizing the model of
  objective function.
\newblock Technical report, Universit\'e catholique de Louvain, Center for
  Operations Research and Econometrics ´ (CORE), 2015.

\bibitem{LP66}
E.~Levitin and B.~T. Polyak.
\newblock Constrained minimization methods.
\newblock {\em USSR Computational Mathematics and Mathematical Physics},
  6(5):787–823, 1966.

\bibitem{GH15}
D.~Garber and E.~Hazan.
\newblock Faster rates for the {F}rank-{W}olfe method over strongly-convex
  sets.
\newblock In {\em Proceedings of the 32nd International Conference on Machine
  Learning (ICML 2015)}, page 541–549, 2015.

\bibitem{Wol70}
J.~Abadie, editor.
\newblock {\em Integer and Nonlinear Programming}.
\newblock North-Holland, Amsterdam, 1970.

\bibitem{GM86}
J.~Guelat and P.~Marcotte.
\newblock Some comments on {W}olfe's `away step'.
\newblock {\em Mathematical Programming}, 35(1):110–119, 1986.

\bibitem{GH13}
D.~Garber and E.~Hazan.
\newblock A linearly convergent conditional gradient algorithm with
  applications to online and stochastic optimization.
\newblock {\em arXiv:1301.4666}, 2013.

\bibitem{LJJ13}
S.~Lacoste-Julien and M.~Jaggi.
\newblock An affine invariant linear convergence analysis for {F}rank-{W}olfe
  algorithms.
\newblock {\em arXiv:1312.7864v2}, 2014.

\bibitem{BS15}
A.~Beck and S.~Shtern.
\newblock Linearly convergent away-step conditional gradient for non-strongly
  convex functions.
\newblock {\em arXiv: 1504.05002}, 2015.

\bibitem{LAN13}
G.~Lan.
\newblock The complexity of large-scale convex programming under a linear
  optimization oracle.
\newblock {\em arXiv:1309.5550}, 2013.

\bibitem{LWM15}
J.~Lafond, H.~Wai, and E.~Moulines.
\newblock Convergence analysis of a stochastic projection-free algorithm.
\newblock {\em arXiv:1510.01171}, 2015.

\bibitem{LH16}
H.~Luo and E.~Hazan.
\newblock Variance-reduced and projection-free stochastic optimization.
\newblock In {\em Proceedings of the 33rd International Conference on Machine
  Learning (ICML 2016)}, volume~48, 2016.

\bibitem{OG10}
H.~Ouyang and A.~Gray.
\newblock Fast stochastic {F}rank-{W}olfe algorithms for nonlinear {SVM}s.
\newblock In {\em SDM}, page 245–256, 2010.

\bibitem{LJJSP13}
S.~Lacoste-Julien, M.~Jaggi, M.~Schmidt, and P.~Pletscher.
\newblock Block-coordinate {F}rank-{W}olfe optimization for structural svms.
\newblock {\em In Proceedings of the 30th International Conference on Machine
  Learning (ICML-13)}, 28:53–61, 2013.

\bibitem{OALDLJ16}
A.~Osokin, J.~B. Alayrac, I.~Lukasewitz, P.~K. Dokania, and S.~Lacoste-Julien.
\newblock Minding the gaps for block {F}rank-{W}olfe optimization of structured
  svms.
\newblock {\em arXiv:1605.09346}, 2016.

\bibitem{MZWG15}
C.~Mu, Y.~Zhang, J.~Wright, and D.~Goldfarb.
\newblock Scalable robust matrix recovery: {F}rank-{W}olfe meets proximal
  methods.
\newblock {\em arXiv:1403.7588}, 2015.

\bibitem{LZ14}
X.~Lin and T.~Zhang.
\newblock A proximal stochastic gradient method with progressive variance
  reduction.
\newblock {\em SIAM Journal on Optimization}, 24(4):2057--2075, 2014.

\bibitem{JAGGI13}
Martin Jaggi.
\newblock Revisiting frank-wolfe: Projection-free sparse convex optimization.
\newblock In {\em ICML (1)}, pages 427--435, 2013.

\bibitem{LJJ15}
Simon Lacoste-Julien and Martin Jaggi.
\newblock On the global linear convergence of {F}rank-{W}olfe optimization
  variants.
\newblock In {\em Advances in Neural Information Processing Systems}, pages
  496--504, 2015.

\bibitem{VW96}
Aad~W. van~der Vaart and Jon~A. Wellner, editors.
\newblock {\em Weak Convergence and Empirical Processes With Applications to
  Statistics}.
\newblock Springer, 1996.

\bibitem{JZ13}
Rie Johnson and Tong Zhang.
\newblock Accelerating stochastic gradient descent using predictive variance
  reduction.
\newblock In C.J.C. Burges, L.~Bottou, M.~Welling, Z.~Ghahramani, and K.Q.
  Weinberger, editors, {\em Advances in Neural Information Processing Systems
  26}, pages 315--323, 2013.

\bibitem{Lich13}
M.~Lichman.
\newblock {UCI} machine learning repository, 2013.

\bibitem{BM11}
Thierry Bertin-Mahieux, Daniel~P.W. Ellis, Brian Whitman, and Paul Lamere.
\newblock The million song dataset.
\newblock In {\em {Proceedings of the 12th International Conference on Music
  Information Retrieval ({ISMIR} 2011)}}, 2011.

\bibitem{duchi08}
John Duchi, Shai Shalev-Shwartz, Yoram Singer, and Tushar Chandra.
\newblock Efficient projections onto the l 1-ball for learning in high
  dimensions.
\newblock In {\em Proceedings of the 25th international conference on Machine
  learning}, pages 272--279. ACM, 2008.

\end{thebibliography}
\end{document}